\documentclass[12pt]{article}
\usepackage{amssymb}
\usepackage{amsmath,acronym,bm,natbib,graphics,mathrsfs,graphicx,epsfig,subfigure,placeins,indentfirst,setspace,enumerate,enumitem,caption,varioref,booktabs,tabularx,color,epstopdf,multirow}
\setcitestyle{authoryear,open={(},close={)}}
\usepackage{amsthm}
\makeatletter \oddsidemargin  -.1in \evensidemargin -.1in
\begin{document}
	\newcommand{\bea}{\begin{eqnarray}}
		\newcommand{\eea}{\end{eqnarray}}
	\newcommand{\nn}{\nonumber}
	\newcommand{\bee}{\begin{eqnarray*}}
		\newcommand{\eee}{\end{eqnarray*}}
	\newcommand{\lb}{\label}
	\newcommand{\nii}{\noindent}
	\newcommand{\ii}{\indent}
	\newtheorem{theorem}{Theorem}[section]
	\newtheorem{example}{Example}[section]
	\newtheorem{corollary}{Corollary}[section]
	\newtheorem{definition}{Definition}[section]
	\newtheorem{lemma}{Lemma}[section]
	\newtheorem{remark}{Remark}[section]
	\newtheorem{proposition}{Proposition}[section]
	\numberwithin{equation}{section}
	\renewcommand{\qedsymbol}{\rule{0.7em}{0.7em}}
	\renewcommand{\theequation}{\thesection.\arabic{equation}}
	%\bibpunct[, ]{(}{)}{;}{a}{,}{,}
	\renewcommand\bibfont{\fontsize{10}{12}\selectfont}
	\setlength{\bibsep}{0.0pt}
	%\doublespacing
		\title{\bf Weighted past and paired dynamic varentropy measures, their properties, usefulness and inferences** }
	
\author{ Shital {\bf Saha}\thanks {Email address: shitalmath@gmail.com} ~and  Suchandan {\bf  Kayal}\thanks {Email address ~(corresponding author):
		~kayals@nitrkl.ac.in,~~suchandan.kayal@gmail.com
		\newline**It has been accepted on \textbf{Journal of Statistical Computation and Simulation}.}
	\\{\it \small Department of Mathematics, National Institute of
			Technology Rourkela, Rourkela-769008, Odisha, India}}
\date{}
\maketitle
\maketitle
		\begin{center}
\textbf{Abstract}
		\end{center} 
%\noindent {\bf Abstract:}			
%\\
%\\
We introduce two uncertainty measures, say weighted past varentropy  (WPVE) and weighted paired dynamic varentropy (WPDVE).  Several  properties of these proposed measures, including their effect under the monotone transformations are studied.  An upper bound of the WPVE using the weighted past Shannon entropy and a lower bound of the WPVE are obtained. Further, the WPVE is studied for the proportional reversed hazard rate (PRHR) models. Upper and lower bounds of the WPDVE are derived. In addition, the non-parametric kernel estimates of the WPVE and WPDVE are proposed. Furthermore, the maximum likelihood estimation technique is employed to estimate WPVE and WPDVE for an exponential population. A numerical simulation is provided to observe the behaviour of the proposed estimates. A real data set is analysed, and then the estimated values of WPVE are obtained. Based on the bootstrap samples generated from the real data set, the performance of the non-parametric and parametric estimators of the WPVE and WPDVE is compared in terms of the absolute bias and mean squared error (MSE). Finally, we have reported an application of WPVE.
		 \\
			\textbf{MSCs:} 94A17; 60E15; 62B10.

\section{Introduction}
Consider a non-negative and absolutely continuous random variable (RV) $Y$. Denote by $g(\cdot)$ the probability density function (PDF) of $Y.$ The information content (IC) and weighted IC of $Y$ are 
\begin{eqnarray}
I(Y)=-\log \big(g(Y)\big) ~~\text{and}~~I^\omega(Y)=-\omega(Y)\log \big(g(Y)\big),
\end{eqnarray}
 respectively, where $\omega(\cdot)>0$ is called the weight function and `log' is a natural logarithm. In literature, $I(Y)$ and $I^{\omega}(Y)$ are also dubbed as the Shannon IC and weighted Shannon IC, respectively. The rationale behind $I(Y)$ can be provided in a discrete scenario, where it signifies the quantity of bits that are fundamentally needed to represent $Y$ through a coding scheme that reduces the average code length.  For details, please refer to \cite{shannon1948mathematical}.  The expectation of $I(Y)$, termed  as Shannon entropy (SE) has been widely studied by many authors (see for example, \cite{hammer2000inequalities}, \cite{kharazmi2021informational}, \cite{saha2024different}) in different fields of research. The SE of $Y$, also known as the differential entropy is defined as
\begin{eqnarray}\label{eq1.2}
\mathcal{H}(Y)=E[I(Y)]=-\int_{0}^{\infty}g(y)\log \big(g(y)\big)dy.
\end{eqnarray}
For a discrete RV $Y$, taking values $y_i$ with respective probabilities $p_i>0,$ $\sum_{i=1}^{n}p_i=1,$  $i=1,\dots,n$,  the SE  is given by
\begin{align}\label{eq1.3}
\mathcal{H}(Y)=-\sum_{i=1}^{n}p_i\log (p_i).
\end{align}
For details,  please refer to \cite{shannon1948mathematical}. Note that the SE measures uncertainty or disorder contained in an RV $Y$. The amount of information and entropy are inter-related. Higher entropy and disorder are correlated with increased information; lower entropy and disorder are correlated with decreased information. Clearly, (\ref{eq1.3}) depends only on the probabilities of occurrence of outcomes. Thus, (\ref{eq1.3}) is not useful in many fields, dealing with experiments where it is required to consider both probabilities and  qualitative characteristic of the events of interest. Thus, for distinguishing the outcomes $y_1,\dots,y_n$ of a goal-directed experiment according to
their importance with respect to a given qualitative characteristic of the system, it is required
to assign numbers $\omega_k> 0$ to each outcome $y_k$. One may choose $\omega_k$, proportional to the importance of the kth outcome. Here, $\omega_k$'s are known as the weights of the outcomes $y_k, ~k = 1,\dots,n$. This type experiment is called as a weighted probabilistic experiment. For such kind of experiments the weighted SE is useful, which is defined as
\begin{eqnarray}\label{eq1.4}
\mathcal{H}^\omega(Y)=-\sum_{i=1}^{n}\omega_ip_i\log( p_i).
\end{eqnarray}
The continuous analogue of (\ref{eq1.4}), known as the weighted SE or weighted differential entropy of the RV $Y$ with weight function $\omega(y)>0,$ is defined as
\begin{eqnarray}\label{eq1.5}
\mathcal{H}^\omega(Y)=-\int_{0}^{\infty}\omega(y)g(y)\log \big(g(y)\big)dy=E[I^\omega(Y)].
\end{eqnarray}
For details, see  \cite{di2007weighted}. Note that $\mathcal{H}^\omega(X)$ in (\ref{eq1.5}) is the expectation of weighted IC.  It is a measurement of the uncertainty and information provided by a probabilistic experiment, which has been used to provide answers to many problems. The SE and weighted SE are used in various fields of areas such as computer science, electrical engineering, behavioural science, environmental science, chemical engineering and in coding theory (see \cite{cover1991elements}). However, there is a discrimination that SE  in (\ref{eq1.2}) is a shift independent measure whenever the weighted SE in (\ref{eq1.5}) is shift dependent. As a result, weighted SE measure is more flexible than SE.

Several researchers grow their interest to study  the behaviour of  the IC which  is useful in  probability, statistics and information theory. The IC concentrates around the SE in higher dimension with the log-concave PDF function, which is studied by \cite{bobkov2011concentration}. Occasionally, the SEs of two RVs have the same value. For example, the SE of the exponential distribution with rate parameter $e$ and uniform distribution in $(0,1)$ are same. In this situation, the idea of the concentration of IC around SE is helpful for analytical explanation.  This concentration can be obtained as the variance of $I(Y)$, which is known as the varentropy (VE). For a non-negative absolutely continuous RV $Y$, the VE (see \cite{fradelizi2016optimal}) is expressed as
\begin{eqnarray}\label{eq1.6}
\mathcal{VE}(Y)=Var[I(Y)]=\int_{0}^{\infty}g(y)[\log \big(g(y)\big)]^{2}dy-\left[\mathcal{H}(Y)\right]^2,
\end{eqnarray} 
 where $\mathcal{H}(Y)$ is the SE of $Y$. Note that $\mathcal{VE}(Y)$ quantifies variability of  $I(Y)$. For a discrete RV $Y$, the VE is given by (see \cite{di2021analysis})
\begin{align}\label{eq1.7}
\mathcal{VE}(Y)=\sum_{i=1}^{n}p_i[\log( p_i)]^2-\left[\sum_{i=1}^{n}p_i\log (p_i)\right]^2.
\end{align} 	
 One of the early appearances of the varentropy is when it was characterised as the ``minimal coding variance" studied by
\cite{kontoyiannis1997second}. Further, \cite{kontoyiannis2013optimal}, used the concept of varentropy as ``dispersion" in source coding in computer science. \cite{maadani2020new} introduced generalised varentropy based on Tsallis entropy and showed that the Tsallis residual varentropy is independent of the age of the systems. \cite{maadani2022varentropy} proposed a method for calculating the varentropy for order statistics and studied some stochastic comparisons. \cite{sharma2023varentropy} introduced the concept of VE in a doubly truncated RV. The authors examined several theoretical properties. \cite{alizadeh2023varentropy} introduced some  non-parametric estimates of the VE with some theoretical properties. They compare the estimates based on the MSEs. 

In survival analysis, the concept of  the residual life is very useful for life testing studies. 
%It helps in determining maintenance when an asset might fail of need maintenance in engineering as this information is very needful for optimizing maintenance schedules, minimizing downtime and maximizing asset utilization. In finance, the residual lifetime uses to the planning knowing of investments which helps in making informed decisions about portfolio management and risk assessment. 
Residual life-based informational measures are also useful for predictive maintenance and decision-making in various fields like reliability engineering, medicine science and finance. \cite{di2021analysis} proposed residual varentropy (RVE) based on  residual lifetime of a system, $Y_t=[Y-t|Y>t]$, $t>0$ with PDF $g_t(y)=\frac{g(y)}{\bar G(t)}$, where $\bar G(t)=P[Y>t]$ represents the reliability function of $Y$. The RVE of $Y_t$ is defined as 
\begin{align}\label{eq1.8}
\mathcal{VE}(Y;t)=Var[I(Y_t)]
= \int_{t}^{\infty}\frac{g(y)}{\bar{G}(t)}\Big(\log\Big(\frac{g(y)}{\bar{G}(t)}\Big)\Big)^2dy-[\mathcal{H}(Y;t)]^2,
\end{align}			
where $\mathcal{H}(Y;t)$ is the residual SE (see \cite{ebrahimi1995new}).  They discussed several mathematical properties and  provided two applications pertaining to the first-passage timings of an Ornstein-Uhlenbeck jump-diffusion process and the proportional hazards model.

The past lifetime occurs when we have failure before a specified inspection time $t>0$. In many situations, it is necessary to measure uncertainty contained in the past lifetime. For example,  in forensic sciences and other related fields, the past lifetimes are used to analyse the right-censored data (see \cite{andersen2012statistical}).  Several researchers studied the uncertainty	for past lifetime in information theory. See, for instance \cite{di2002entropy}, \cite{di2007weighted}, \cite{di2021fractional}, and \cite{saha2023extended}. Recently, \cite{buono2022varentropy} introduced varentropy of the past lifetime. Let $G(t)=P[Y<t]$ be the cumulative distribution function (CDF) of $Y$. The past lifetime of a system is denoted by $Y^*_t=[t-Y|Y\le t].$ The PDF of $Y^*_t$ is $g^*_t(y)=\frac{g(y)}{G(t)}$. The VE of the past lifetime is defined as
\begin{align}\label{eq1.9}
\mathcal{VE^*}(Y;t)= \int_{0}^{t}\frac{g(y)}{ G(t)}\Big(\log\Big(\frac{g(y)}{ G(t)}\Big)\Big)^2dy-[\mathcal{H}^*(Y;t)]^2,
\end{align}	
where $\mathcal{H}^*(Y;t)$ is the past SE (see \cite{di2002entropy}). Note that $\mathcal{VE^*}(Y;t)$ in (\ref{eq1.9}) is the variance of the IC, $I(Y^*_t)=-\log \big(\frac{g(Y)}{G(t)}\big)$.  \cite{raqab2022varentropy} considered past VE and obtained some reliability properties associated with the past VE. \cite{sharma2024stochastic} introduced various theoretical properties of the past VE.

Very recently, \cite{saha2024weighted} proposed weighted varentropy (WVE) for discrete as well as continuous RVs, and examined some properties. The authors also studied WVE of the coherent systems. They further proposed weighted residual varentropy (WRVE). The WVE is  the variance of the weighted IC, $I^\omega(Y)=-\omega(Y)\log (g(Y))$. \cite{saha2024weighted} showed that the WVE gives better result than the VE for different distributions. For a discrete RV $Y$, the WVE is given by
\begin{align}
\mathcal{VE}^\omega(Y)=\sum_{i=1}^{n}\omega^2_ip_i[\log (p_i)]^2-\left[\sum_{i=1}^{n}\omega_ip_i\log( p_i)\right]^2,
\end{align} 		
where $p_i$'s and $\omega_i$'s are probability mass function and weight function corresponding to the event $Y=y_i,$ for $i=1,\dots,n$. Analogously, the WVE of $Y$ is defined as 
\begin{align}
\mathcal{VE}^\omega(Y)=\int_{0}^{\infty}\omega^2(y)g(y)\Big(\log \big(g(y)\big)\Big)^2dy-[\mathcal{H}^\omega(Y)]^2,
\end{align}
where $\mathcal{H}^\omega(Y)$ is the  weighted SE of $Y$ (see \cite{di2007weighted}). It is clear that VE as well as WVE are non-negative. For uniform distribution, the value of VE is zero but WVE is non-zero. The WRVE of $Y_t$ is defined as (see \cite{saha2024weighted})
\begin{align}\label{eq1.7}
\mathcal{VE}^\omega(Y;t)=Var[IC^\omega(Y_t)]
= \int_{t}^{\infty}\frac{g(y)}{\bar G(t)}\bigg(\omega(y)\log\Big(\frac{g(y)}{\bar G(t)}\Big)\bigg)^2dy-[\mathcal{H}^\omega(Y;t)]^2,
\end{align}			
where $\mathcal{H}^w(Y;t)$ is the weighted residual SE (see \cite{di2007weighted}) with weight $\omega(y)>0$. The authors have proposed non-parametric estimator of the WRVE. Further, they have illustrated the proposed estimate using a simulation study and two real data sets.  In this communication, motivated by the aforementioned findings and the usefulness of the weight function in probabilistic  experiment, we introduce weighted past varentropy (WPVE) and study its various properties. In the following, the key contributions of this paper are discussed.
\begin{itemize}
	\item In Section \ref{sec2}, we propose weighted varentropy for the past lifetime. This measure is called as the WPVE. The proposed measure is a generalisation of  the varentropy, weighted varentropy and past varentropy. The WPVE is studied under a monotonically transformed RVs. Lower and upper bounds of the WPVE are obtained. Further, in Section \ref{sec3} the WPVE  is studied for the PRHR model.
	
	\item In Section \ref{sec4}, the concepts of weighted paired dynamic entropy (WPDE) and WPDVE are introduced. Several bounds of the WPDVE are obtained. The effect of the WPDVE under an affine transformation is examined. 
	
	\item In Section \ref{sec5}, the kernel-based non-parametric estimates of the WPVE and WPDVE are proposed. To see their performance, a Monte Carlo simulation study is carried out. For both WPVE and WPDVE, we have further considered parametric estimation assuming that the data are taken from an exponential population. Average daily wind speeds data set is considered and analysed. It is observed that the parametric estimates have superior performance over the non-parametric estimates in terms of the absolute bias (AB) and MSE values. 
	
	\item In Section \ref{sec6}, an application of WPVE  related to the  reliability engineering using coherent systems is provided. Finally, the conclusion of the work has been discussed in Section \ref{sec7}.
\end{itemize}

Henceforth, we assume that the RVs are non-negative and absolutely continuous unless it is mentioned. Further, `increasing' and `decreasing' are used in wide sense. The differentiation, integration and expectation always exist wherever they are used.

\section{Weighted past varentropy}\label{sec2}
In this section, we introduce an information measure by taking the variance of the weighted IC for the past lifetime $Y^*_t$. The weighted IC of $Y^*_t$ is $I^\omega(Y^*_t)=-\omega(y)\log(\frac{g(y)}{G(t)}),$ where $\omega(y)>0$ is the weight function.
% First, we consider the following definition.

\begin{definition}
Let $Y$ have the  CDF $G(\cdot)$ and PDF $g(\cdot)$. The WPVE is defined by 
 \begin{align}\label{eq2.1}
 \mathcal{\overline{VE}^\omega}(Y;t)=Var[I^\omega(Y^*_t)]
 = \int_{0}^{t}\frac{g(y)}{ G(t)}\bigg(\omega(y)\log\Big(\frac{g(y)}{ G(t)}\Big)\bigg)^2dy-[\overline{\mathcal{H}}^\omega(Y;t)]^2,
 \end{align}
 where $\overline{\mathcal{H}}^\omega(Y;t)$ is the weighted past SE (see  \cite{di2007weighted}).
\end{definition}	

\begin{remark}
	The WPVE can be considered as a generalisation of the weighted VE (see \cite{saha2024weighted}) and past VE (see \cite{buono2022varentropy}). In particular, (\ref{eq2.1}) reduces to the past VE when $\omega(y)=1$, while  (\ref{eq2.1}) becomes the weighted VE for $t\rightarrow\infty.$ Further, when $\omega(y)=1$ and $t\rightarrow\infty$, then the WPVE coincides with the VE (\cite{fradelizi2016optimal}). 
\end{remark}
	
For the weight function $\omega(y)=y,$ (\ref{eq2.1}) can be written as 
\begin{align}\label{eq2.2}
\mathcal{\overline{VE}}^y(Y;t)=E[(\psi_{1}(Y))^2|Y\le t]-2\Lambda^*(t)\overline{\mathcal{H}}^{y^2}(Y;t)-(\Lambda^*(t))^2E[Y^2|Y\le t]-[\overline{\mathcal{H}}^{y}(Y;t)]^2,
\end{align}
where $\psi_{1}(y)=y\log(g(y))$. In (\ref{eq2.2}),  $\Lambda^*(t)=-\log (G(t))$ is the cumulative reversed hazard rate (CRHR)  function and $\overline{\mathcal{H}}^{y^2}(Y;t)$ is the weighted past SE with weight $y^2$. Now, we obtain the closed form expression of WPVE.
\begin{example}~~\label{ex2.1}
\begin{itemize}
\item[$(i)$] Suppose the uniform RV $Y$ has the CDF $G(y)=\frac{y-a}{b-a}$, $y\in[a,b]$. Then, the WPVE of $Y$ is
\begin{align*}
\mathcal{\overline{VE}}^{y}(Y;t)=\frac{1}{12}\{\log(t-a)\}^2\{4(t^2+at+a^2)-3(t+a)^2\},
\end{align*}
plotted in Figure \ref{fig1}$(a)$ to see its behaviour with respect to $t>0.$
\item[$(ii)$] For Pareto-I  distribution with CDF $G(y)=1-y^{-\alpha},~y\ge1,~\alpha>0$, the WPVE is obtained as
\begin{align*}
\mathcal{\overline{VE}}^{y}(Y;t)=&\frac{\psi_2(t;\alpha)t^{2-\alpha}}{(2-\alpha)^3}\bigg[\Big\{1+\alpha+\log\big(\psi^{2-\alpha}_2(t;\alpha)\big)\Big\}^2+(1+\alpha)^2\Big\{\log(t^{2-\alpha})-1\Big\}^2\\
&-\log\big(t^{2(1+\alpha)(2-\alpha)^2}\big)\bigg]-\frac{\psi^2_2(t;\alpha)t^{2-2\alpha}}{(1-\alpha)^2}\bigg\{\log\big(\psi^{1-\alpha}_2(t;\alpha)\big)-\log\big(t^{(1-\alpha)(1-\alpha^2)}\big)\\&+\alpha^2-1\bigg\}^2,
\end{align*}
where $ \psi_{2}(t;\alpha)=\frac{\alpha}{1-t^{-\alpha}}.$ For graphical plot of the WPVE of Pareto-I  distribution with respect to $t$, see Figure \ref{fig1}$(b)$.
\item[$(iii)$] Consider an exponential RV $Y$ with CDF $G(y)=1-e^{-\lambda y}, ~y>0,~\lambda>0$. Then, the WPVE is obtained as
\begin{align*}
\mathcal{\overline{VE}}^{y}(Y;t)=&\psi_3(t;\lambda)\bigg[\Big\{\log\big(\psi_3(t;\lambda)\big)-3\Big\}^2\Big\{\frac{2}{\lambda^3}\big(1-e^{-\lambda t}(1+\lambda t)\big)-\frac{t^2e^{-\lambda t}}{\lambda}\Big\}\\
&+\frac{t^2e^{-\lambda t}}{\lambda}\Big\{2\lambda t \log\big(\psi_3(t;\lambda)\big)-4\lambda t-\lambda^2 t^2-3\Big\}+\frac{12}{\lambda^3}\Big\{1-e^{-\lambda t}(1+\lambda t)\Big\}\bigg]\\
&-\frac{1}{\lambda^2(1-e^{-\lambda t})^2}\bigg[\Big\{1-e^{-\lambda t}(1+\lambda t)\Big\}\Big\{\log\big(\psi_3(t;\lambda)\big)-2\Big\}+\lambda^2 t^2e^{-\lambda t}\bigg]^2,
\end{align*}
where $\psi_{3}(t;\lambda)=\frac{\lambda}{1-e^{-\lambda t}}.$ The plot of the WPVE with respect to $t$ is provided in Figure \ref{fig1}$(c)$.
\end{itemize}
\end{example}

\begin{figure}[h!]
		\centering
	\subfigure[]{\label{c1}\includegraphics[height=1.9in]{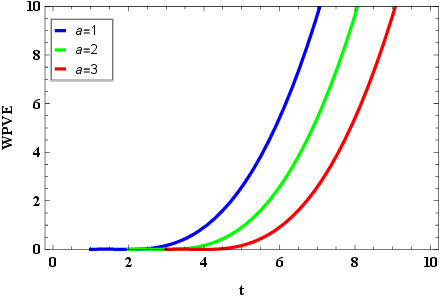}}
	\subfigure[]{\label{c1}\includegraphics[height=1.9in]{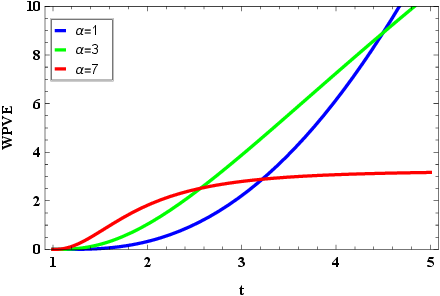}}
	\subfigure[]{\label{c1}\includegraphics[height=1.9in]{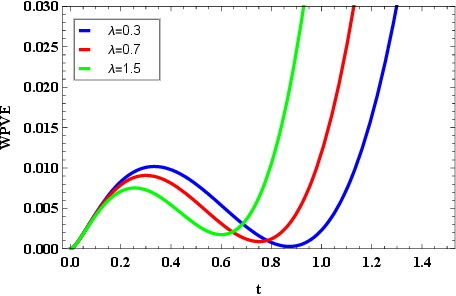}}
%	\subfigure[]{\label{c1}\includegraphics[height=1.9in]{CExE3.eps}}
		\caption{Graphs for the WPVE of $(a)$ uniform distribution in Example \ref{ex2.1}$(i)$, $(b)$  Pareto-I distribution in Example \ref{ex2.1}$(ii)$, and $(c)$ exponential distribution in Example \ref{ex2.1}$(iii)$.} 
		\label{fig1}
	\end{figure}

Bounds in probability (e.g., Markov's inequality, Chebyshev's inequality) provide limits of the likelihood of events. This is important for understanding the spread and distribution of random variables. In addition, bounds also help in estimating probabilities of information measures, particularly when exact calculations are complex or infeasible.
Below, we obtain an upper bound of the WPVE via weighted past SE and CRHR function. %cumulative reversed hazard rate function.
 \begin{theorem}\label{th2.1}
 Suppose $Y$ is an RV with PDF $g(\cdot).$ Further, let the PDF satisfy
 \begin{align}\label{eq2.3}
 e^{-(\alpha y+\beta )}\le g(y)\le1,~~y>0,~ \alpha>0,~\beta\ge0.
 \end{align}
Then, for $t>0$
 \begin{align}
 \mathcal{\overline{VE}}^{y}(Y;t)\le \overline{\mathcal{H}}^{\omega_2}(Y;t)-2\Lambda^*(t)E[\alpha Y^3+\beta Y^2|Y\le t]+\Lambda^{*^2}(t)E[Y^2|Y\le t],
 \end{align}
 where $\overline{\mathcal{H}}^{\omega_2}(Y;t)$ is the weighted past SE with weight $\omega_2(y)=\alpha y^3+\beta y^2$.
 \end{theorem}
 \begin{proof}
 For $\omega(y)=y$, from  (\ref{eq2.1}) we obtain
 \begin{align}\label{eq2.4}
  \mathcal{\overline{VE}}^{y}(Y;t)=\int_{0}^{t}\frac{g(y)}{ G(t)}\bigg(y\log\left(\frac{g(y)}{ G(t)}\right)\bigg)^2dy-[\overline{\mathcal{H}}^{y}(Y;t)]^2
  \le \int_{0}^{t}y^2\frac{g(y)}{ G(t)}\left[\log (g(y))+\Lambda^*(t)\right]^2dy.
  \end{align}
  Further, 
  \begin{align*}
  \int_{0}^{t}y^2\frac{g(y)}{ G(t)}[\log (g(y))+\Lambda^*(t)]^2dy&=\int_{0}^{t}y^2\frac{g(y)}{ G(t)}[\log (g(y))]^2dy+2\int_{0}^{t}y^2\frac{g(y)}{ G(t)}\log (g(y)) \Lambda^*(t)dy\\
  &+\int_{0}^{t}y^2\frac{g(y)}{ G(t)}[\Lambda^*(t)]^2dy\\
  &\le -\int_{0}^{t}y^2(\alpha y+\beta)\frac{g(y)}{ G(t)}\log (g(y))dy+[\Lambda^*(t)]^2\int_{0}^{t}y^2\frac{g(y)}{ G(t)}dy\\
  &-2\int_{0}^{t}y^2(\alpha y+\beta)\frac{g(y)}{ G(t)} \Lambda^*(t)dy\\
  &=\overline{\mathcal{H}}^{\omega_2}(Y;t)-2\Lambda^*(t)E[(\alpha Y^3+\beta Y^2)|Y\le t]\\
 & +\Lambda^{*^2}(t)E[Y^2|Y\le t].
 \end{align*}
 \end{proof}
 
 In the following example, we show that Lomax distribution satisfies the condition in (\ref{eq2.3}).
 
 \begin{example}
 Consider the Lomax distribution with CDF $G(x)=1-(1+\frac{x}{\delta})^{-\gamma},~x>0,~\delta>0,$ and $\gamma>0$.  The inequality in (\ref{eq2.3}) can be easily checked from the graphical plots (see Figure \ref{fig2}). 
 \end{example}

\begin{figure}[h!]
		\centering
		\includegraphics[width=13.5cm,height=8cm]{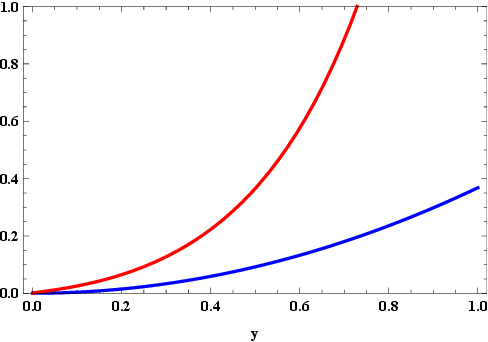}
		%\subfigure[]{\label{c1}\includegraphics[height=2in]{com1.eps}}
	%\subfigure[]{\label{c1}\includegraphics[height=1.9in]{mon2.pdf}}
	%\subfigure[]{\label{c1}\includegraphics[height=1.9in]{WPVE_exp1.eps}}
%	\subfigure[]{\label{c1}\includegraphics[height=1.9in]{CExE3.eps}}
		\caption{Graphical plots of $e^{-(\alpha x+\beta)}$ (blue colour) and $g(x)=\frac{\gamma}{\delta}(1+x/\delta)^{-(\gamma+1)}$ (red colour) for $\alpha=2, ~\beta=1,~ \delta=1,$ and $\gamma=3$. To capture the full support $x\in(0,\infty)$, we take $x=-\log y$, where $y\in(0,1).$ }
	\label{fig2}	
	\end{figure}

 Variance of past lifetime (VPL) is an important concept used in reliability theory and survival analysis. It provides valuable information about the variability of the past lifetime of a system or individual, given that it has already stopped working, inspected at time $t>0$. By studying the variance of past lifetime, organizations can develop maintenance strategies that minimize costs and avoid failures. In warranty analysis, it helps in determining the likelihood and variability of failures during the warranty period, aiding in better warranty design.
 For an RV $Y$,  the VPL is defined as 
 \begin{align}
 \sigma^2(t)=Var[t-Y|Y\le t]=\frac{2}{G(t)}\int_{0}^{t}du\int_{u}^{t}g(z)dz-[\mathcal{M}(t)]^2,
 \end{align} 
 where $\mathcal{M}(t)=\int_{0}^{t}\frac{G(y)}{G(t)}dy$ is called the mean past lifetime (MPL). For detailed study on VPL, please see \cite{mahdy2016further}.
 Now, we obtain a lower bound of the WPVE in terms of the  VPL.
\begin{theorem}\label{th2.2}
We have
	\begin{eqnarray}
	\mathcal{\overline{VE}}^\omega(Y;t)\geq \sigma^2(t)\{1+E[-\zeta_t(Y_t)\log (g_t(Y_t))]+E[Y_t\zeta^{'}_t(Y_t)]\}^2,
	\end{eqnarray}
	where $\zeta_t(\cdot)$ can be determined from
	\begin{align}\label{eq2.8}
		\sigma^2(t)\zeta_t(y)g_t(y)=\int_{0}^{y}(\mathcal{M}(t)-u)g_t(u)du, ~y>0.
	\end{align}
	
\end{theorem}
\begin{proof}
	Let $Y$ be an RV with PDF $g(\cdot)$, mean $m,$ and variance $\sigma^2$. Then,  
	\begin{eqnarray}\label{eq2.9}
	Var[\mathcal{I}(Y)]\geq \sigma^2(E[\eta(Y)\mathcal{I}^{'}(Y)])^2,
	\end{eqnarray}
	where $\eta(.)$ can be obtained using  $\int_{0}^{y}(m-u)g(u)du=\sigma^2\eta(y)g(y)$ (see \cite{cacoullos1989characterizations}). For proving the required result, we consider $Y_t$ as a reference RV with $\mathcal{I}(y)=I^{y}(y)=-y\log (g(y))$. From (\ref{eq2.9}), we get
	\begin{eqnarray}\label{eq2.10}
		Var[-Y_t\log (g_t(Y_t))]&\geq& \sigma^2(t)\bigg\{E\big[\zeta_t(Y_t)\big(-Y_t\log (g_t(Y_t))\big)'\big]\bigg\}^2\nonumber\\
		&=&\sigma^2(t)\bigg\{E[-\zeta_t(Y_t)\log (g_t(Y_t))]-E\bigg[\zeta_t(Y_t)Y_t\frac{g^{'}_t(Y_t)}{g_t(Y_t)}\bigg]\bigg\}^2.
		\end{eqnarray}
		Further,
		\begin{eqnarray}\label{eq2.11}
		E\bigg[\zeta_t(Y_t)Y_t\frac{g^{'}_t(Y_t)}{g_t(Y_t)}\bigg]= E\bigg[Y_t\bigg(\frac{\mathcal{M}(t)-Y_t}{\sigma^2(t)}-\zeta^{'}_t(Y_t)\bigg)\bigg]= -1-E[Y_t\zeta^{'}_t(Y_t)].
		\end{eqnarray}
		Using (\ref{eq2.11}) in (\ref{eq2.10}), the required result can be easily obtained. 
\end{proof}

Next, we present a corollary. Its proof readily follows from Theorem \ref{th2.2}, and thus it is omitted.
\begin{corollary}~
	\begin{itemize}
		\item[(i)] Suppose $\zeta_t(y)$ is increasing function in $y>0$. Then, 
		$$\overline{\mathcal{VE}}^\omega(Y;t)\geq\sigma^2(t)\big(E[\zeta_t(Y_t)\log (g_t(Y_t))]\big)^2.$$
		\item[(ii)] Let $g_t(y)\leq1.$ Then,
		$$\overline{\mathcal{VE}}^\omega(Y;t)\geq\sigma^2(t)\big(E[Y_t\zeta^{'}_t(Y_t)]\big)^2.$$
	\end{itemize}
\end{corollary}

Sometimes, it is hard to evaluate the closed-form expression of the WPVE for a transformed RV. The following theorem is useful to obtain the WPVE of a new distribution constructed using a monotone transformation. Note that the monotone transformations are useful tools that preserve entropy or information of an RV in the field information theory. 

\begin{theorem}\label{th2.3}
Let $Y$ be an RV and $X=\psi(Y)$, where $\psi$ is a strictly monotonic, continuous and differentiable function. Then, 
\begin{equation}\label{eq3.20}
		\overline{\mathcal{VE}}^x(X;t)=\left\{
		\begin{array}{ll}
			 \overline{\mathcal{VE}}^{\psi}(Y;\psi^{-1}(t)) -2\overline {\mathcal{H}}^{\psi}(Y;\psi^{-1}(t))E[\gamma_1(Y)|Y\leq\psi^{-1}(t)]\\
			+Var[\gamma_1(Y)|Y\leq\psi^{-1}(t)]-2E\Big[\psi(Y)\gamma_1(Y)\log\Big( \frac{g(Y)}{G(\psi^{-1}(t))}\Big)\Big|Y\leq\psi^{-1}(t)\Big],\\~	
			if~\psi~ is ~strictly~ increasing;
			\\
			\\
			{\mathcal{VE}}^{\psi}(Y;\psi^{-1}(t)) -2\mathcal{H}^{\psi}(Y;\psi^{-1}(t))E[\gamma_2(Y)|Y>\psi^{-1}(t)]\\
			+Var[\gamma_2(Y)|Y>\psi^{-1}(t)]-2E\Big[\psi(Y)\gamma_2(Y) \log\Big( \frac{g(Y)}{G(\psi^{-1}(t))}\Big)\Big|Y>\psi^{-1}(t)\Big], \\~if~ \psi~ is ~strictly~ decreasing,
		\end{array}
		\right.
	\end{equation}
	where $\gamma_1(y)=\psi(y)\log(\psi'(y))$,  $\gamma_2(y)=\psi(y)\log(-\psi'(y))$,  $\mathcal{H}^{\psi}(Y;t)=-\int_{t}^{\infty}\psi(y)\frac{g(y)}{\bar{G}(t)}\log \big(\frac{g(y)}{\bar{G}(t)}\big)dy$, $\mathcal{VE}^{\psi}(Y;t)=\int_{t}^{\infty}\psi^2(y)\frac{g(y)}{\bar{G}(t)}\left(\big(\frac{g(y)}{\bar{G}(t)}\big)\right)^2dy-[\mathcal{H}^{\psi}(Y;t)]^2$,
	$\overline {\mathcal{H}}^{\psi}(Y;t)=-\int_{0}^{t}\psi(y)\frac{g(y)}{{G}(t)}\log \big(\frac{g(y)}{{G}(t)}\big)dy$, 
	and
	$\overline{\mathcal{VE}}^{\psi}(Y;t)=\int_{0}^{t}\psi^2(y)\frac{g(y)}{{G}(t)}\Big(\log \big(\frac{g(y)}{{G}(t)}\big)\Big)^2dy-[\overline {\mathcal{H}}^{\psi}(Y;t)]^2$.
\end{theorem}
\begin{proof}
	Assume that $\psi$ is a strictly increasing function. Now,
\begin{align}\label{eq2.13}
\overline {\mathcal{H}}^{x}(X;t)=&-\int_{0}^{\psi^{-1}(t)}\psi(y)\frac{g(y)}{G(\psi^{-1}(t))}\log\Big(\frac{g(y)}{G(\psi^{-1}(t))}(\psi^{'}(y))^{-1}\Big)dy\nonumber\\
=&-\int_{0}^{\psi^{-1}(t)}\psi(y)\frac{g(y)}{G(\psi^{-1}(t))}\log\Big(\frac{g(y)}{G(\psi^{-1}(t))}\Big)dy\nonumber\\
&+\int_{0}^{\psi^{-1}(t)}\psi(y)\frac{g(y)}{G(\psi^{-1}(t))}\log\big(\psi^{'}(y)\big)dy\nonumber\\
=&{\mathcal{H}}^{\psi}(Y;\psi^{-1}(t))+E\big[\psi(Y)\log(\psi^{'})(Y)|Y\le\psi^{-1}(t)\big].
\end{align}	
From (\ref{eq2.1}), we further have
\begin{align}\label{eq2.14*}
\overline{\mathcal{VE}}^{x}(X;t)=\int_{0}^{\psi^{-1}(t)}\psi^2(y)\frac{g(y)}{G(\psi^{-1}(t))}\Big\{\log\Big(\frac{g(y)}{G(\psi^{-1}(t))}(\psi^{'}(y))^{-1}\Big)\Big\}^2dy-\big(\overline {\mathcal{H}}^{x}(X;t)\big)^2.\nonumber\\
\end{align}	
Furthermore,
\begin{align}\label{eq2.15*}
&\int_{0}^{\psi^{-1}(t)}\psi^2(y)\frac{g(y)}{G(\psi^{-1}(t))}\Big\{\log\Big(\frac{g(y)}{G(\psi^{-1}(t))}(\psi^{'}(y))^{-1}\Big)\Big\}^2dy\nonumber\\
&=\int_{0}^{\psi^{-1}(t)}\psi^2(y)\frac{g(y)}{G(\psi^{-1}(t))}\Big\{\log\Big(\frac{g(y)}{G(\psi^{-1}(t))}\Big)-\log\big(\psi^{'}(y)\big)\Big\}^2dy\nonumber\\
&=\int_{0}^{\psi^{-1}(t)}\psi^2(y)\frac{g(y)}{G(\psi^{-1}(t))}\Big\{\log\Big(\frac{g(y)}{G(\psi^{-1}(t))}\Big)\Big\}^2dy\nonumber\\
&+\int_{0}^{\psi^{-1}(t)}\psi^2(y)\frac{g(y)}{G(\psi^{-1}(t))}\Big\{\log\Big(\psi^{-1}(t)\Big)\Big\}^2dy\nonumber\\
&-2\int_{0}^{\psi^{-1}(t)}\psi^2(y)\frac{g(y)}{G(\psi^{-1}(t))}\log\big(\psi^{-1}(t)\big)\log\Big(\frac{g(y)}{G(\psi^{-1}(t))}\Big)dy\nonumber\\
&=\int_{0}^{\psi^{-1}(t)}\psi^2(y)\frac{g(y)}{G(\psi^{-1}(t))}\Big\{\log\Big(\frac{g(y)}{G(\psi^{-1}(t))}\Big)\Big\}^2dy+E\big[\psi^2(Y)\big(\log \psi^{'}(Y)\big)^2|Y\le\psi^{'}(t)\big]\nonumber\\
&-2E\Big[\psi^2(Y)\log(\psi^{'}(Y))\log\Big(\frac{g(Y)}{G(\psi^{'}(t))}\Big)\Big|Y\le\psi^{-1}(t)\Big].
\end{align}		
	
Using (\ref{eq2.13}) and (\ref{eq2.15*}) in (\ref{eq2.14*}), we obtain			
	\begin{align}\label{eq2.15}
		\overline{\mathcal{VE}}^x(X;t)&=\int_{0}^{\psi^{-1}(t)}\psi^2(y)\frac{g(y)}{G(\psi^{-1}(t))}\Big\{\log\Big(\frac{g(y)}{G(\psi^{-1}(t))}\Big)\Big\}^2dy\nonumber\\
		&+E\big[\psi^2(Y)\big(\log \psi^{'}(Y)\big)^2|Y\le\psi^{'}(t)\big]\nonumber\\
		&-2E\Big[\psi^2(Y)\log(\psi^{'}(Y))\log\Big(\frac{g(Y)}{G(\psi^{'}(t))}\Big)\Big|Y\le\psi^{-1}(t)\Big]\nonumber\\
		&-\{\mathcal{H}^{\psi}(Y;\psi^{-1}(t))+E[\psi(Y)\log (\psi^{'}(Y))|Y\le\psi^{-1}(t)]\}^2\nonumber\\
		&=\overline{\mathcal{VE}}^{\psi}(Y;\psi^{-1}(t)) -2\overline {\mathcal{H}}^{\psi}(Y;\psi^{-1}(t))E[\gamma_1(Y)|Y\leq\psi^{-1}(t)]\nonumber\\
		&+Var[\gamma_1(Y)|Y\leq\psi^{-1}(t)]-2E\Big[\psi(Y)\gamma_1(Y)\log\Big( \frac{g(Y)}{G(\psi^{-1}(t))}\Big)\Big|Y\leq\psi^{-1}(t)\Big].
	\end{align}
	Thus, the proof is made for strictly increasing function $\psi$. The proof for strictly decreasing function $\psi$ is similar, and thus it is omitted. This completes the proof.  	 
\end{proof}

To see the validation of the result in Theorem \ref{th2.3}, we consider the following example. 

\begin{example}\label{ex2.3}
Consider exponential RV $Y$ with the CDF $G_1(y)=1-e^{-\lambda y},~y>0$ and $\lambda>0$. Further, let $X=\psi(Y)=Y^2$, which is strictly increasing, continuous and differentiable function. Here, $X$ follows  Weibull distribution with CDF $G_2(y)=1-e^{-\lambda y^{\frac{1}{2}}},~y>0$ and $\lambda>0$. Then, the WPVE of $X=\psi(Y)=Y^2$ is obtained as
\begin{align}\label{eq2.14}
\overline{\mathcal{VE}}^x(X;t)&=\frac{\lambda}{1-e^{-\lambda\sqrt{t}}}\Big[t\sqrt{t}\Big(2t\log\Big(\frac{\lambda}{1-e^{-\lambda\sqrt{t}}}\Big)-\lambda\sqrt{t}-6\Big)e^{-\lambda\sqrt{t}}+\frac{1}{\lambda^5}\Big\{24-\Big(24(1+\lambda\sqrt{t})\nonumber\\
&+12\lambda^2t+4\lambda^3t\sqrt{t}\lambda^4t^2\Big)e^{-\lambda\sqrt{t}}\Big\}\Big]-\phi_{1}(\lambda;t)\Big\{\phi_{1}(\lambda;t)+2E[Y^2\log(2Y)|Y\le\sqrt{t}]\Big\}\nonumber\\
&+\text{Var}[Y^2\log(2Y)|Y\le\sqrt{t}]-2E\Big[Y^4\log(2Y)\log\Big(\frac{\lambda e^{-\lambda Y}}{1-e^{-\lambda\sqrt{t}}}\Big)\Big|Y\le\sqrt{t}\Big],
\end{align}
where $\phi_{1}(\lambda;t)=\frac{1}{\lambda^2(1-e^{-\lambda\sqrt{t}})}\Big[\{2-(\lambda^2t+2\lambda\sqrt{t}+2)e^{-\lambda\sqrt{t}}\}\log\Big(\frac{\lambda}{1-e^{-\lambda\sqrt{t}}}\Big)+(6+6\lambda\sqrt{t}+3\lambda^2t+\lambda^3t\sqrt{t})e^{-\lambda\sqrt{t}}-6\Big]$. We have plotted the graphs of WPVE in (\ref{eq2.14}) in Figures \ref{fig3} $(a)$ and $(b)$ with respect to $t$ (for fixed $\lambda$) and with respect to $\lambda$ (for fixed $t$), respectively.
\end{example}

\begin{figure}[h!]
		\centering
	\subfigure[]{\label{c1}\includegraphics[height=1.9in]{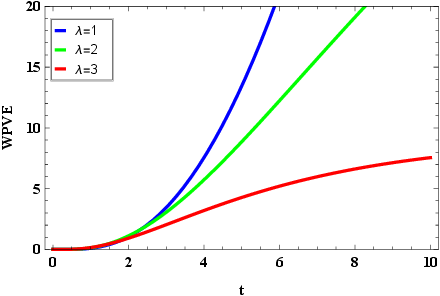}}
	\subfigure[]{\label{c1}\includegraphics[height=1.9in]{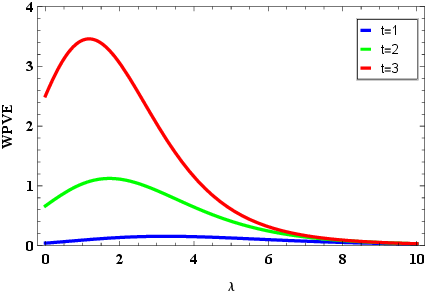}}
	%\subfigure[]{\label{c1}\includegraphics[height=1.9in]{WPVE_exp1.eps}}
%	\subfigure[]{\label{c1}\includegraphics[height=1.9in]{CExE3.eps}}
		\caption{ Plots of the WPVE for the Weibull distribution in Example \ref{ex2.3} $(a)$ with respect to $t$ (for fixed $\lambda$) and  $(b)$ with respect to $\lambda$ (for fixed $t$).}
	\label{fig3}	
	\end{figure}

Now, we investigate the effect of the WPVE under the affine transformation $X=\alpha Y+\beta$, $\alpha>0$ and $\beta\ge0$.
\begin{corollary}\label{th2.4}
Suppose $X$ is an RV and $X=\alpha Y+\beta$ with $\alpha>0,~\beta\ge0.$ Then, 
\begin{align*}
\overline{\mathcal{VE}}^{x}(X;t)=&\overline{\mathcal{VE}}^{\omega_1}\Big(Y;\frac{t-\beta}{\alpha}\Big)+\big(\log (\alpha)\big)^2\text{Var}\left[\alpha Y+\beta\Big|Y\le \frac{t-\beta}{\alpha}\right]\\&-2\log (a)\left(\overline{\mathcal{H}}^{\omega_1}\Big(Y;\frac{t-\beta}{\alpha}\Big)+\mathcal{\overline{H}}^{\omega_1}\Big(Y;\frac{t-\beta}{\alpha}\Big)E\bigg[\alpha Y+\beta\Big|Y\le \frac{t-\beta}{\alpha}\bigg]\right),
\end{align*}
where $\overline{\mathcal{VE}}^{\omega_1}\left(Y;\frac{t-\beta}{\alpha}\right)$  and  $\mathcal{\overline{H}}^{\omega_1}\left(Y;\frac{t-\beta}{\alpha}\right)$ are the WPVE and  weighted past SE with weight $\omega_1(y)=\alpha y+\beta$, respectively.
\end{corollary}
\begin{proof}
	We here omit the proof, since it readily follows from Theorem \ref{th2.3}. 
%The random variable of $Y$ has CDF $G(x)=F(\frac{x-b}{a})$ and PDF $g(x)=\frac{1}{a}f(\frac{x-b}{a})$.
% Then, the weighted cumulative SE of $Y$ is
% \begin{align}\label{eq2.12}
% \overline{\mathcal{H}}^\omega(Y;t)&=-\int_{0}^{t}x\frac{g(x)}{G(t)}\log \Big(\frac{g(x)}{G(t)}\Big)dx\nonumber\\
% &=-\int_{0}^{\frac{t-b}{a}}\Big(\frac{x-b}{a}\Big)\frac{f(x)}{F(\frac{t-b}{a})}\left\{\log \Big(\frac{f(x)}{F(\frac{t-b}{a})}\Big)-\log (a)\right\}dx\nonumber\\
% &=\overline{\mathcal{H}}^{\omega_1}\Big(X;\frac{t-b}{a}\Big)+\log (a) E\left[\frac{X-b}{a}\Big|X\le \frac{t-b}{a}\right].
% \end{align}
% and
% \begin{align}\label{eq2.13}
% \int_{0}^{t}x^2\frac{g(x)}{G(t)}\log^2 \Big(\frac{g(x)}{G(t)}\Big)dx&=\int_{0}^{\frac{t-b}{a}}\Big(\frac{x-b}{a}\Big)^2\frac{f(x)}{F(\frac{t-b}{a})}\bigg\{\log \Big(\frac{f(x)}{F(\frac{t-b}{a})}\Big)-\log (a)\bigg\}^2dx\nonumber\\
% &=\int_{0}^{\frac{t-b}{a}}\Big(\frac{x-b}{a}\Big)^2\frac{f(x)}{F(\frac{t-b}{a})}\left\{\log \Big(\frac{f(x)}{F(\frac{t-b}{a})}\Big)\right\}^2dx\nonumber\\
% &-2\log (a) \overline{\mathcal{H}}^{\omega_1}\Big(X;\frac{t-b}{a}\Big)+\big(\log (a)\big)^2E\left[\left(\frac{X-b}{a}\right)^2\Big| X\le \frac{t-b}{a}\right].
% \end{align}
% Using (\ref{eq2.12}) and (\ref{eq2.13}) in (\ref{eq2.1}), the result is easily follows. This completes the proof.
\end{proof}	

Next, an example is considered to illustrate Corollary \ref{th2.4}.

\begin{example}\label{ex2.4}
Consider the CDF $G(y)=1-e^{-y},~y>0$. We take $X=Y+\beta,~\beta>0$. Thus, from Corollary \ref{th2.4}, the WPVE of $X$ is obtained as
\begin{align*}
\overline{\mathcal{VE}}^{x}(X;t)&=
\frac{1}{\psi_{3}(t;\beta)}\bigg[\Big\{2\log(\psi_{3}(t;\beta))+(6+2\beta)e^{-\beta}\Big\}\Big\{\beta^5-(t-2\beta)^5e^{-(t-3\beta)}\Big\}\\
&+\Big\{\Big(\log\big(\psi_{3}(t;\beta)\big)\Big)^2
+2(5+\beta)\log(\psi_{3}(t;\beta))+5(6+2\beta)e^{-\beta}\Big\}\Big\{\beta^4+4(\beta^2-4\beta\\
&+6)e^{-\beta}-(t-2\beta)^4e^{-(t-3\beta)}
-(t-2\beta)^4e^{-(t-3\beta)}-4(3-\beta)e^{-(t-4\beta)}\Big((t-2\beta+1)^2\\
&-(t-2\beta)^3\Big)\Big\}\bigg]-\big(\psi_{4}(t;\beta)\big)^2,
\end{align*}
where $\psi_{3}(t;\beta)=e^{-(t-\beta)}-1$, $\psi_{4}(t;\beta)=\frac{1}{\psi_{3}(t;\beta)}\Big[(\beta^2-4\beta+6)-e^{-(t-3\beta)}\Big\{(3-\beta)(t-2\beta+1)^2+(3-\beta)-(t-2\beta)^3\Big\}+\log(\psi_{3}(t;\beta))\Big\{(\beta^2-2\beta+2)+\big((t-2\beta+1)^2+1\big)e^{-(t-3\beta)}\Big\}\Big]$ and $\psi_{5}(t;\beta)=1-\beta+(2\beta-t-1)e^{-(t-\beta)}$. We have plotted the graphs of the WPVE of $X$ in Figure \ref{fig4}. Clearly, the WPVE is non-monotone with respect to $t$ and $\beta.$
\end{example}

\begin{figure}[h!]
		\centering
	\subfigure[]{\label{c1}\includegraphics[height=2.0in]{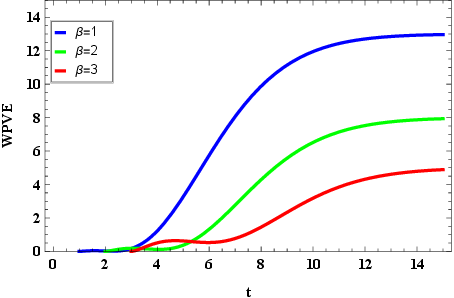}}
	 \subfigure[]{\label{c1}\includegraphics[height=2.0in]{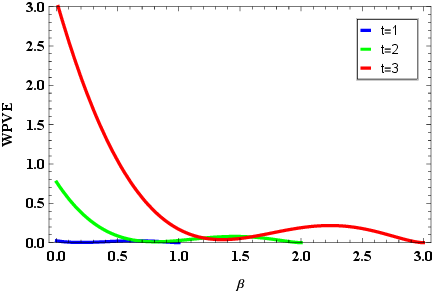}}
	%\subfigure[]{\label{c1}\includegraphics[height=1.9in]{WPVE_exp1.eps}}
%	\subfigure[]{\label{c1}\includegraphics[height=1.9in]{CExE3.eps}}
		\caption{Graphs of the WPVE in Example \ref{ex2.4} $(a)$ with respect to $t$ (for fixed $b$)  and $(b)$ with respect to $\beta$ (for fixed $t$).}
		\label{fig4}
	\end{figure}

\section{WPVE for PRHR model}\label{sec3}
The PRHR model is a crucial idea to use in reliability engineering, survival analysis, industries, and  various other fields. \cite{gupta1998modeling} introduced PRHR model and discussed its properties. It offers flexibility in modelling the reversed hazard rate of events that exhibit decreasing reversed failure rates over time. This is particularly useful in situations where the traditional models (like the Weibull distribution) which assumes increasing or constant hazard rates, may not adequately capture the behaviour of the data.
 Let $Y$ and $X$ be two RVs with CDFs $G_1(\cdot)$ and $G_2(\cdot)$, respectively. The PRHR model of RVs $X$ and $Y$ is defined by
\begin{align}\label{eq3.1}
G_2(t)=P[Y^{(a)}\le t]=[G_1(t)]^a,~~\text{for all $t>0$ and $a>0$.}
\end{align} 
Its PDF is
\begin{align}
g_2(t)=a[G_1(t)]^{a-1}g_1(t),
\end{align} 
where $g_1(\cdot)$ and $g_2(\cdot)$ are the PDFs of $Y$ and $X$, respectively.  It is known as the Lehmans's alternatives for $a>0$. Various researchers studied PRHR model and introduced its several properties. \cite{gupta2007proportional} discussed Fisher information in PRHR model. \cite{li2008mixture} proposed a mixture model of PRHR model and discussed some properties. 
%\cite{kizilaslan2017bayesian} discussed the E-Bayesian and hierarchical Bayesian estimations for the proportional reversed hazard rate model based on record values.
 For some properties of the PRHR model, we refer to \cite{finkelstein2002reversed}, \cite{nanda2011dynamic}, \cite{balakrishnan2018necessary}, and \cite{popovic2021generalized}. Suppose $\tilde\Lambda(x)$ denotes the CRHR function  of $X$, is defined by
\begin{eqnarray}\label{eq3.2}
\tilde\Lambda(y)=-\log \big({G_2}(y)\big)=-\log \big([{G_1}(y)]^a\big)=a\Lambda^{*}(y),~y>0.
\end{eqnarray}
The weighted past SE of $X$ is obtained as
\begin{eqnarray}\label{eq3.4}
\overline{\mathcal{H}}^y(X;t)&=& -\int_{0}^{t}y\frac{g_2(y)}{{G_2}(t)}\log \big(g_2(y)\big)dy-\int_{0}^{t}y\frac{g_2(y)}{{G_2}(t)}\tilde\Lambda(t)dy\nonumber\\
&=&- \frac{1}{[ {G_1}(t)]^a}\bigg\{\int_{0}^{[{G_1}(t)]^{a}}\mathcal{L}(x:a)dx+a\Lambda^*(t)\int_{0}^{[{G_1}(t)]^{a}} {G_1}^{-1}(x^{1/a})dx\bigg\}\nonumber\\
&=&\frac{1}{[ {G_1}(t)]^a}\int_{0}^{[{G_1}(t)]^{a}}\mathcal{J}(x:a,t)dx,
\end{eqnarray}
where $x=[{G_1}(y)]^a$, $\mathcal{L}(y:a)= {G_1}^{-1}(x^{1/a})\log \big(ax^{1-1/a}g[ {G_1}^{-1}(x^{1/a})]\big),$ and 
\begin{eqnarray}\label{eq3.5}
\mathcal{J}(x:a,t)= -{G_1}^{-1}(x^{1/a})\log \Big(a\frac{y^{1-1/a}}{[{G_1}(t)]^a}g_1[ {G_1}^{-1}(x^{1/a})]\Big),
\end{eqnarray}
where $ {G_1}^{-1}(x^{1/a})=\sup\{y:{G_1}(y)\le x^{1/a}\}$ is called the quantile function of ${G_1(\cdot)}$. Next, we obtain WPVE for the PRHR model in (\ref{eq3.1}).

\begin{theorem}\label{th3.1}
	Suppose  $X$ is an RV having CDF $G_2(\cdot)$ in (\ref{eq3.1}) . Then, for $a>0$ and $t>0$, the WPVE  of $X$  is
	\begin{eqnarray}\label{eq3.6}
	\overline{\mathcal{VE}}^y(X;t)=\frac{1}{[{G_1}(t)]^a}\int_{0}^{[{G_1}(t)]^{a}}[\mathcal{J}(x:a,t)]^2dx-\frac{1}{[{G_1}(t)]^{2a}}\bigg\{\int_{0}^{[{G_1}(t)]^{a}}\mathcal{J}(x:a,t)dx\bigg\}^2,
	\end{eqnarray}
	where $\mathcal{J}(x:a,t)$ is given in (\ref{eq3.5}).
\end{theorem}

\begin{proof}
	 Using (\ref{eq3.1}) in (\ref{eq2.1}), we have 
	\begin{eqnarray}\label{eq3.7}
\overline{\mathcal{VE}}^y(X;t)=\int_{0}^{t}\frac{g_2(y)}{{G_2}(t)}\bigg(y\log \Big(\frac{g_2(y)}{ {G_2}(t)}\Big) \bigg)^2dy-[\overline{\mathcal{H}}^y(X;t)]^2.
	\end{eqnarray}
	Now, using $x=[{G_1}(y)]^a$
	\begin{eqnarray}\label{eq3.8}
	\int_{0}^{t}\frac{g_2(y)}{{G_2}(t)}\bigg(y\log \Big(\frac{g_2(y)}{ {G_2}(t)}\Big) \bigg)^2dy&=&\frac{1}{[{G_1}(t)]^a}\bigg\{\int_{0}^{[{G_1}(t)]^{a}}[\mathcal{L}(x:a)]^2dx\nonumber\\
	&~&+2\Lambda^{*(a)}(t)\int_{0}^{[{G_1}(t)]^{a}}{G_1}^{-1}(x^{1/a})\mathcal{L}(x:a)dx\nonumber\\
	&~& +(\Lambda^{*(a)}(t))^2\int_{0}^{[{G_1}(t)]^{a}}[{G_1}^{-1}(x^{1/a})]^2dx\bigg\}\nonumber\\
	&=& \frac{1}{[{G_1}(t)]^a}\int_{0}^{[{G_1}(t)]^{a}}[\mathcal{J}(x:a,t)]^2dx.
	\end{eqnarray}
	Using (\ref{eq3.4}) and (\ref{eq3.8}) in (\ref{eq3.7}), the result readily follows. 
\end{proof}

%An another closed form of $\mathcal{J}(y:a,t)$ by using the concept of reversed hazard rate function and cumulative reversed hazard rate function can be expressed as
%\begin{eqnarray}\label{eq5.5}
%\mathcal{J}(y:a,t)=(a\Lambda^{*})^{-1}\log (y)\log\left[a y e^{a\Lambda^*(t)}r\left\{(\Lambda^{*})^{-1}\Big(-\frac{1}{a}\log (y)\Big)\right\}\right].
%\end{eqnarray}

Next, using Theorem \ref{th3.1}, we obtain the WPVE for a PRHR model.
\begin{example}\label{ex3.1}
Consider the power distribution with CDF $G_1(y)=(\frac{y}{\beta})^\alpha,~y>0,~\alpha>0$ and $\beta>0$. The WPVE for PRHR model with baseline distribution as power distribution is 
\begin{align}\label{eq3.9}
\overline{\mathcal{VE}}^{y}(X;t)&=\frac{a\alpha t^2}{\beta^2(2+a\alpha)}\bigg[\Big\{\log\bigg(\Big(\frac{a\alpha\beta^{a\alpha-1}}{t^{a\alpha}}\Big)^\beta\bigg)\Big\}^2+2\beta(a\alpha-1)\Big(\log(t/\beta)-\frac{1}{2+a\alpha}\Big)\nonumber\\
~&\times \log\bigg(\Big(\frac{a\alpha\beta^{a\alpha-1}}{t^{a\alpha}}\Big)^\beta\bigg)+\frac{2\beta^2(a\alpha-1)^2}{(2+a\alpha)}\Big\{(1+a\alpha)\log(t/\beta)-\frac{1}{2+a\alpha}\Big\} \bigg]\nonumber\\
~&-\Big(\frac{\beta}{t}\Big)^{2a\alpha}\Big(\frac{t^{1+a\alpha}}{\beta^{a\alpha(1+a\alpha)}}\Big)^2\bigg[a\alpha\Big\{\log(a\alpha/\beta)-a\alpha\log(t/\beta)\Big\}+\frac{a\alpha-1}{1+a\alpha}\bigg]^2.
\end{align}
We have plotted the WPVE in (\ref{eq3.9}) in Figure \ref{fig5} with respect to $t$ and $a$. 

\end{example}

\begin{figure}[h!]
		\centering
	\subfigure[]{\label{c1}\includegraphics[height=2in]{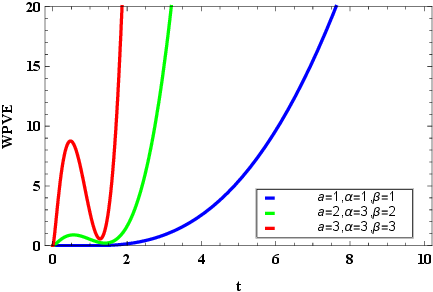}}
	\subfigure[]{\label{c1}\includegraphics[height=2in]{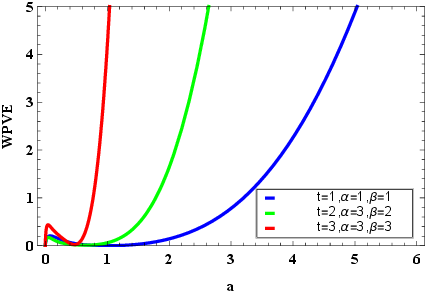}}
	%\subfigure[]{\label{c1}\includegraphics[height=1.9in]{WPVE_exp1.eps}}
%	\subfigure[]{\label{c1}\includegraphics[height=1.9in]{CExE3.eps}}
		\caption{ Plots for the WPVE  $(a)$ with respect to $t$ (for fixed $a,\alpha,\beta$) and $(b)$ with respect to $a$ (for fixed $t,\alpha,\beta$) in Example \ref{ex3.1}.}
	\label{fig5}	
	\end{figure}

\section{Weighted paired dynamic varentropy}\label{sec4}
Let $Y$ be a discrete RV with mass function $p_{i}$, $i=1,\ldots,n$. Then, the paired entropy is (see \cite{burbea1982convexity})
 \begin{align}
 \mathcal{PH}(Y)=-\sum_{i=1}^{n}\Big[(1-p_i)\log (1-p_i)+p_i\log (p_i)\Big].
  \end{align}
Motivated by the paired entropy, \cite{klein2016cumulative} introduced cumulative paired entropy for a continuous RV, which is given by 
 \begin{align}\label{eq4.2}
	\mathcal{CP}(Y)=-\int_{0}^{\infty}[G(y)\log(G(y))+\bar{G}(y)\log(\bar{G}(y))]dy.
\end{align}

Note that $\mathcal{CP}$ in (\ref{eq4.2}) is a combination of the cumulative entropy (see \cite{di2009cumulative}) and cumulative residual entropy (see \cite{rao2004cumulative}). Further, \cite{klein2016cumulative} studied its properties. In particular, they discussed how cumulative paired entropy used directly or implicitly working in five scientific disciplines: Fuzzy set theory, generalised maximum entropy principle, theory of dispersion of ordered categorical variables, uncertainty theory and reliability theory   with an entropy based on distribution functions or survival functions. Motivated by the concept of the paired entropy and cumulative paired entropy, here, we introduce a new information measure combining the concept of past entropy and residual entropy. Suppose $Y$ is an RV with CDF $G(\cdot)$. Then,  the WPDE for $t>0$ is defined as 
\begin{align}
\mathcal{PH}^\omega(Y;t)&=-\left[\int_{0}^{t}\omega(y)\frac{g(y)}{G(t)}\log \left(\frac{g(y)}{G(t)}\right)dy+\int_{t}^{\infty}\omega(y)\frac{g(y)}{\bar G(t)}\log \left(\frac{g(y)}{\bar G(t)}\right)dy\right]\nonumber\\
&=\overline{\mathcal{H}}^\omega(Y;t)+\mathcal{H}^\omega(Y;t).
\end{align}	

Consider an affine transformation  $X=\alpha Y+\beta$, where $\alpha>0$, $\beta\ge0$. Then, for $\omega(y)=y$, the WPDE is obtained as
\begin{equation}\label{eq4.4*}
\mathcal{PH}^y(X;t)=\mathcal{PH}^{\omega_2}\Big(Y;\frac{t-\beta}{\alpha}\Big)+\log(\alpha)\Big\{E\Big[\alpha Y+\beta\Big|Y\le\frac{t-\beta}{\alpha}\Big]+E\Big[\alpha Y+\beta\Big|Y\ge\frac{t-\beta}{\alpha}\Big]\Big\},
\end{equation} 
where $\omega_2(y)=\alpha y+\beta.$ From (\ref{eq4.4*}), we observe that like weighted dynamic (residual and past) entropies, the WPDE is also shift-dependent. Next, the closed form expression of the WPDE is obtained. 

\begin{example}~~\label{ex4.1}
	\begin{itemize}
		\item[$(i)$] For the uniform RV $Y$ with CDF $G(y)=\frac{y}{\beta}, ~y\in[0,\beta]$ and $\beta>0$, the WPDE is 
		\begin{align}
			\mathcal{PH}^y(Y;t)&=\frac{t}{2}\log (t)+\frac{\beta+t}{2}\big(\log(\beta-t)\big),~t>0.
		\end{align}
		\item[$(ii)$] Assume that $Y$ follows exponential distribution with mean $1/\lambda$. For $t>0,$
		\begin{align}
			\mathcal{PH}^y(Y;t)&=\frac{1}{\lambda(e^{-\lambda t}-1)}\Big[\Big\{1-e^{-\lambda t}(1+\lambda t)\Big\}\Big\{\log \Big(\frac{\lambda}{1-e^{-\lambda t}}\Big)-2\Big\}+\lambda^2t^2e^{-\lambda (t)}\Big]\nonumber\\
			&+\frac{1}{\lambda}\Big\{(\lambda t+1)\log(\lambda)-\lambda t-2\Big\}.
		\end{align}
	To see the behaviour of the WPDE for uniform and exponential distributions, we have plotted their WPDEs in Figure 	\ref{fig6}. The graphs show that the WPDE is non-monotone for these distributions.	
	\end{itemize}
	\end{example}

\begin{figure}[h!]
	\centering
	\subfigure[]{\label{c1}\includegraphics[height=2in]{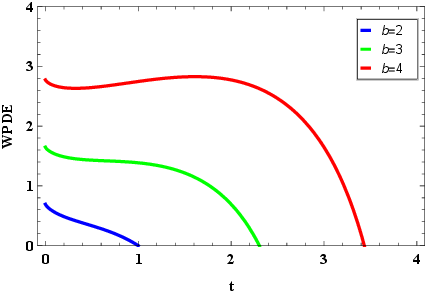}}
	\subfigure[]{\label{c1}\includegraphics[height=2in]{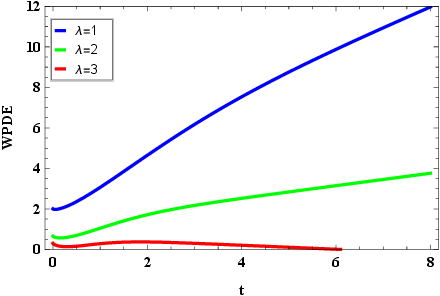}}
	%\subfigure[]{\label{c1}\includegraphics[height=1.9in]{WPVE_exp1.eps}}
	%	\subfigure[]{\label{c1}\includegraphics[height=1.9in]{CExE3.eps}}
	\caption{ Plots of the WPDEs  for $(a)$  the uniform distribution (for fixed $b$) and $(b)$ the exponential distribution (for fixed $\lambda$) in Example \ref{ex4.1}$(i)$ and Example \ref{ex4.1}$(ii)$, respectively.}
	\label{fig6}	
\end{figure}

Inspired by the notions of the paired, cumulative paired, and weighted paired dynamic entropies, here we propose the concept of the WPDVE.
\begin{definition}\label{def4.1}
	Suppose $Y$ is an RV with CDF $G(\cdot)$,  PDF $g(\cdot),$ and survival function $\bar G(\cdot)$. The  WPDVE of $Y$ is defined as
	\begin{align}\label{eq4.3}
		\mathcal{PVE}^\omega(Y;t)
%		=\text{Var}(IC_v^\omega)&=Var \left(-xlog \frac{f^2(x)}{F(t)\bar F(t)}\right)\nonumber\\
%		&=Var\left(-x\log \frac{f(x)}{ F(t)}\right)+Var\left(-x\log \frac{f(x)}{\bar F(t)}\right)\nonumber\\
		=\overline{\mathcal{VE}}^\omega(Y;t)+\mathcal{VE}^\omega(Y;t),
	\end{align}
	where $\overline{\mathcal{VE}}^\omega(Y;t)$ and $\mathcal{VE}^\omega(Y;t)$ are respectively known as the WPVE (see (\ref{eq2.1})) and WRVE (see \cite{saha2024weighted}). 
\end{definition}
Note that 
\begin{eqnarray}
	\overline{\mathcal{VE}}^\omega(Y;t)+\mathcal{VE}^\omega(Y;t)&=&Var\Big(-\omega(Y)\log (g^*_t(Y))\Big)+Var\Big(-\omega(Y)\log (g_t(Y))\Big)\nonumber\\&=&Var \Big(-\omega(Y)\log (g^*_t(Y)\times g_t(Y))\Big)\nonumber\\
	&=&Var\Big(IC_J^\omega(Y)\Big),
\end{eqnarray}
where $IC_J^\omega(y)=-\omega(y)\log (g^*_t(y)\times g_t(y))$ is the combined IC of past and residual random lifetimes. 
%
%the weighted paired varentropy (WPRVE) can be defined as 
%\begin{align}\label{eq4.3}
%	\mathcal{PVE}^\omega(X;t)=\text{Var}(IC_v^\omega)&=Var \left(-xlog \frac{f^2(x)}{F(t)\bar F(t)}\right)\nonumber\\
%	&=Var\left(-x\log \frac{f(x)}{ F(t)}\right)+Var\left(-x\log \frac{f(x)}{\bar F(t)}\right)\nonumber\\
%	&=\overline{\mathcal{VE}}^\omega(X;t)+\mathcal{VE}^\omega(X;t)
%\end{align}
%where $\mathcal{VE}^\omega$ and $\overline{\mathcal{VE}}^\omega$ are WRVE and WPVE is defined in (\ref{eq1.6}) and (\ref{eq3.1}) respectively.
%
% Now, we discuss another varentropy for the information content $IC_v^\omega=-xlog \frac{f^2(x)}{F(t)\bar F(t)},~\text{for any real number}~ t\ge0$, we call it weighted paired varentropy.
%\begin{definition}
%Suppose $X$ is a non-negative absolutely continuous random variable with CDF $F(\cdot)$ ,  PDF $f(\cdot)$ and its survival function $\bar F(\cdot)$. For the information content $IC_v^\omega=-xlog \frac{f^2(x)}{F(t)\bar F(t)},~\text{any real number}~ t\ge0$, the weighted paired varentropy (WPRVE) can be defined as 
%\begin{align}\label{eq4.3}
%\mathcal{PVE}^\omega(X;t)=\text{Var}(IC_v^\omega)&=Var \left(-xlog \frac{f^2(x)}{F(t)\bar F(t)}\right)\nonumber\\
%&=Var\left(-x\log \frac{f(x)}{ F(t)}\right)+Var\left(-x\log \frac{f(x)}{\bar F(t)}\right)\nonumber\\
%&=\overline{\mathcal{VE}}^\omega(X;t)+\mathcal{VE}^\omega(X;t)
%\end{align}
%where $\mathcal{VE}^\omega$ and $\overline{\mathcal{VE}}^\omega$ are WRVE and WPVE is defined in (\ref{eq1.6}) and (\ref{eq3.1}) respectively.
%\end{definition}
In the following, we express WPDVE in terms of the conditional expectations, WPDE, and weighted dynamic (residual and past) entropies when $\omega(y)=y$:
\begin{align}
\mathcal{PVE}^y(Y;t)=&E[(Y\log(g(Y))^2)|Y<t]+E[(Y\log(g(Y))^2)|Y>t]-[{\mathcal{PH}}^y(Y;t)]^2\nonumber\\
&-2\Lambda(t){\mathcal{H}}^{y^2}(Y;t)-2\Lambda^*(t)\overline{\mathcal{H}}^{y^2}(Y;t)-(\Lambda^*(t))^2E[Y^2|Y\le t]\nonumber\\
&-(\Lambda(t))^2E[Y^2|Y> t]+2{\mathcal{H}}^y(Y;t)\overline{{\mathcal{H}}}^y(Y;t).
\end{align}
When $t\rightarrow 0$ or $t\rightarrow\infty,$ the WPDVE in (\ref{eq4.3}) reduces to the weighted varentropy, which has been studied by \cite{saha2024weighted}. Further, considering $\omega(y)=1$, the WPDVE becomes usual varentropy (see \cite{fradelizi2016optimal}) when $t\rightarrow 0$ or $t\rightarrow\infty.$ Due to these reasons, the newly proposed information measure in Definition \ref{def4.1} can be treated as a generalised information measure. 
Next, we establish a lower bound of the WPDVE via the WPVE and WRVE.

\begin{theorem}
Suppose $Y$ is an RV. Then, for a general weight function $\omega(\cdot)$, we have
\begin{align}\label{eq4.5}
\mathcal{PVE}^\omega(Y;t)\ge \max\{\overline{\mathcal{VE}}^\omega(Y;t), \mathcal{VE}^\omega(Y;t)\},~t>0.
\end{align}
\end{theorem}

\begin{proof}
From (\ref{eq4.3}), it is clear that
\begin{align}\label{eq4.6}
\mathcal{PVE}^\omega(Y;t)\ge\overline{\mathcal{VE}}^\omega(Y;t)>0
\end{align}
and
\begin{align}\label{eq4.7}
\mathcal{PVE}^\omega(Y;t)\ge  \mathcal{VE}^\omega(Y;t)>0.
\end{align}
Now, combining (\ref{eq4.6}) and (\ref{eq4.7}), the result follows.
\end{proof}

The following theorem provides an upper bound of the WPDVE.

\begin{theorem}
For the  RV $Y$
 \begin{align*}
 \mathcal{PVE}^y(Y;t)\le& E[(\psi_{1}(Y))^2|Y\le t]+ E[(\psi_{1}(Y))^2|Y\ge t]-2\Lambda^*(t)\overline{\mathcal{H}}^{y^2}(Y;t)-2\Lambda(t){\mathcal{H}}^{y^2}(Y;t),
 \end{align*}
 where $t>0$ and $\psi_{1}(y)=y\log(g(y))$.
\end{theorem}

\begin{proof}
 From (\ref{eq2.2}), we have
\begin{align}\label{eq4.13*}
\mathcal{\overline{VE}}^y(Y;t)=E[(\psi_{1}(Y))^2|Y\le t]-2\Lambda^*(t)\overline{\mathcal{H}}^{y^2}(Y;t)-(\Lambda^*(t))^2E[Y^2|Y\le t]-[\overline{\mathcal{H}}^{y}(Y;t)]^2.
\end{align}
It is clear that $(\Lambda^*(t))^2E[Y^2|Y\le t]$ and $[\overline{\mathcal{H}}^{y}(Y;t)]^2$ are always non-negative. Using this observation, from (\ref{eq4.13*}) we obtain
\begin{align}\label{eq4.8}
 \mathcal{\overline{VE}}^y(Y;t)\le E[(\psi_{1}(Y))^2|Y\le t]-2\Lambda^*(t)\overline{\mathcal{H}}^{y^2}(Y;t).
 \end{align}
 Further, from $(3.2)$ of \cite{saha2024weighted}, we get an upper bound of the WRVE likewise in (\ref{eq4.8}) as
 \begin{align}\label{eq4.9}
  \mathcal{{VE}}^y(Y;t)\le E[(\psi_{1}(Y))^2|Y\ge t]-2\Lambda(t){\mathcal{H}}^{y^2}(Y;t).
  \end{align} 
Thus, the required bound follows after summing (\ref{eq4.8}) and (\ref{eq4.9}), completing the proof of the theorem.
\end{proof}

\begin{theorem}
Suppose $Y^*_t$ and $Y_t$ respectively denote the past and residual lifetimes with finite MRL $\mu(t)$, finite MPL $\mathcal{M}(t)$, finite VRL $\sigma^2_1(t)$, and VPL $\sigma^2(t)$. Then,
\begin{align}
\mathcal{PVE}^y(Y;t)\ge \max\{\pi(y,t),\theta(y,t)\},
\end{align}
where $\pi(y,t)=\sigma^2(t)\{1+E[-\zeta_t(Y_t)\log \big(g_t(Y_t\big)]+E[Y_t\zeta^{'}_t(Y_t)]\}^2$ and $\theta(y,t)=\sigma^2_1(t)\{1+E[-\eta_t(Y_t)\log \big(g_t(Y_t\big)+E[Y_t\eta^{'}_t(Y_t)]\}^2$. Here, $\eta_t(y)$ and $\zeta_t(y)$ are real-valued functions, respectively obtained from
	$$\sigma^2_1(t)\eta_t(y)g_t(y)=\int_{0}^{y}(\mu(t)-u)g_t(u)du, ~y>0$$
and  
	\begin{align*}
		\sigma^2(t)\zeta_t(y)g^*_t(y)=\int_{0}^{y}(\mathcal{M}(t)-u)g^*_t(u)du, ~y>0.
	\end{align*}	
\end{theorem}

\begin{proof}
From Theorem \ref{th2.2}, we have
\begin{eqnarray}\label{eq4.11}
	\overline{\mathcal{VE}}^y(Y;t)\geq \sigma^2(t)\{1+E[-\zeta_t(Y_t)\log (g^*_t(Y_t))]+E[Y_t\zeta^{'}_t(Y_t)]\}^2.
	\end{eqnarray}
Further, from Theorem $3.3$ of \cite{saha2024weighted}, we get
\begin{eqnarray}\label{eq4.12}
	\mathcal{VE}^y(Y;t)\geq \sigma^2_1(t)\{1+E[-\eta_t(Y_t)\log (g_t(Y_t))]+E[Y_t\eta^{'}_t(Y_t)]\}^2.
	\end{eqnarray}
	Thus, using (\ref{eq4.11}) and (\ref{eq4.12}), the result readily follows.
\end{proof}

This section ends with a result dealing with the effect of the WPDVE under affine transformations. 

\begin{theorem}
Let $Y$ be an RV. Assume that $X=aY+b$ with $a>0,~b\ge0$. Then, for all real number $t> 0$,
\begin{align*}
\mathcal{PVE}^y(X;t)&=\mathcal{PVE}^{\omega_1}(X;t)-2\log (a)\mathcal{PH}^{\omega_1}(X;t)+(\log (a))^2[\xi(a,b,t)\{1-\xi(a,b,t)\}\\
&+\varphi(a,b,t)\{1-\varphi(a,b,t)\}]
-2\log (a) [\mathcal{H}^{\omega_1}(Y;(t-b)/a)\{1+\varphi(a,b,t)\}\\
&+\overline{\mathcal{H}}^{\omega_1}(Y;(t-b)/a)\{1+\xi(a,b,t)\}],
\end{align*}
where $\mathcal{PVE}^{\omega_1}(X;t)$ and $\mathcal{PH}^{\omega_1}(X;t)$ are the WPDVE and WPDE with weight function $\omega_1\equiv\omega_1(y)=ay+b$ respectively, and $\varphi(a,b,t)=E[aY+b|Y>(t-b)/a]$ and $\xi(a,b,t)=E[aY+b|Y\le(t-b)/a]$.
\end{theorem}
\begin{proof}
 From Corollary \ref{th2.4} and Corollary $3.1$ of \cite{saha2024weighted}, the proof follows. 
\end{proof}

\section{Estimation of the WPVE and WPDVE}\label{sec5}

This section presents kernel-based non-parametric estimates of the WPVE and WPDVE. We remark here that the non-parametric estimators are essential because they offer robustness, flexibility, and versatility, making them effective in situations where traditional parametric methods fail due to incorrect assumptions, small sample sizes, or complex data structures. They empower statisticians and data scientists to draw meaningful insights from a wider range of data. We see the performance of the proposed estimates using a Monte-Carlo simulation study. For both WPVE and WPDVE, we have also considered parametric estimation assuming that the data are taken from exponential population. A data set representing average daily wind speeds is considered and analysed for the purpose of estimating WPVE.

\subsection{WPVE }
Here, we consider non-parametric and parametric estimations of the WPVE. First, we discuss about non-parametric estimation.
\subsubsection{Non-parametric estimation}
We introduce a non-parametric estimator based on the kernel estimates of  WPVE in (\ref{eq2.1}). The kernel estimate of the PDF ${g}(\cdot)$ is given by
 \begin{eqnarray}\label{eq5.1}
 \widehat g(y)=\frac{1}{nb_n}\sum_{i=1}^{n}\mathcal{K}\left(\frac{y-Y_i}{b_n}\right), 
 \end{eqnarray} 
where $\mathcal{K}(\cdot)$ is known as kernel, satisfying the following properties.
\begin{itemize}
\item It is non-negative;
\item $\int \mathcal{K}(y)dy=1$;
\item The kernel is symmetric at the origin;
\item It satisfies the Lipschitz condition.
\end{itemize}
%  $(A)$ Lipschitz condition $(B)$ $\int \mathcal{K}(y)dy=1$, and $(C)$  the $\mathcal{K}(\cdot)$ is symmetric with respect to the origin. 
  In (\ref{eq5.1}), the sequence of positive real numbers $\{b_n\}$  is known as the bandwidths such that $b_n\rightarrow0$ and $nb_n\rightarrow\infty,$ for $n\rightarrow\infty$. For details about the kernel density estimates, readers can refer  to \cite{rosenblat1956remarks} and \cite{parzen1962estimation}. Using (\ref{eq5.1}), a non-parametric kernel estimate of $\overline{\mathcal{VE}}^y(Y;t)$ is
\begin{eqnarray}\label{eq5.2}
\widehat{ \overline{\mathcal{VE}}^y}(Y;t)=\int_{0}^{t}\widehat{\eta}(y)(y\log  \widehat{\eta}(y))^2dy-\left[\int_{0}^{t}y\widehat{\eta}(y)\log  \widehat{\eta}(y)dy\right]^2,~t>0,
\end{eqnarray}
where $\widehat{\eta}(y)=\frac{\widehat g(y)}{\widehat G(t)}$ and $\widehat G(t)=\int_{0}^{t}\widehat g(y)dy$. Below, we conduct Monte-Carlo simulation to see the performance of the estimate given in (\ref{eq5.2}).

\subsection*{\textbf{Simulation study}}

A Monte-Carlo simulation study has been performed to generate data sets from exponential distribution with mean $1/\lambda$ for different sample sizes. The true parameter value is taken as $\lambda=0.7$. The software ``Mathematica" has been used for the simulation study. For computing the AB and MSE of the kernel-based non-parametric estimate, we use $100$ replications. Here, Gaussian kernel is used  for estimation. It is given by
\begin{align}\label{eq5.3*}
\mathcal{K}(y)=\frac{1}{\sqrt{2\pi}}e^{-\frac{y^2}{2}}.
\end{align}
The AB and MSE of the kernel-based non-parametric estimate $\widehat{ \overline{\mathcal{VE}}^y}(Y;t)$ in (\ref{eq5.2})  have been computed and presented in Table \ref{tb1} for different choices of $t$ and $n$. We have considered $t=0.1,0.2,0.3,0.4,1$ and $n=100,120, 150,200.$ From Table \ref{tb1}, we observe that in general the AB and MSE  decrease as $n$ increases. This confirms the consistency of the proposed estimate $\widehat{ \overline{\mathcal{VE}}^y}(Y;t)$ in (\ref{eq5.2}). 

\begin{table}[ht!]
\caption {The AB, MSE and $\overline{\mathcal{VE}}^y(Y;t)$ for the kernel-based estimate of WPVE in (\ref{eq5.2}).}
	\centering % used for centering table
	\scalebox{1}{\begin{tabular}{c c c c c c c c } % centered columns (4 columns)
			\hline\hline\vspace{.1cm} %inserts double horizontal lines
			$t$ &$n$ &~~~~ AB &~~~~MSE &~~~~$\overline{\mathcal{VE}}^y(Y;t)$ \\
			\hline\hline
\multirow{10}{1.9cm}
~ & 100 &~~~~  $0.002899$ &~~~~  $0.000065$&~   \\[1.2ex]
~&  120 & ~~~~ $0.002375$ &~~~~  $0.000008$&~   \\[1.2ex]
0.1& 150 & ~~~~ $0.002066$ & ~~~~ $0.000006$&~~~~ $0.004285$  \\[1.2ex]
~ & 200 & ~~~~ $0.001970$ & ~~~~ $0.000005$&~   \\[1.2ex]

\hline
\multirow{10}{1.9cm}
~ & 100 & ~~~~ $0.004743$ & ~~~~ $0.000026$&~  \\[1.2ex]
~&  120 & ~~~~ $0.004657$ & ~~~~ $0.000025$&~   \\[1.2ex]
0.2& 150 & ~~~~ $0.004287$ &~~~~  $0.000022$&~~~~ $0.007901$  \\[1.2ex]
~ & 200 & ~~~~ $0.004317$ &~~~~  $0.000020$&~   \\[1.2ex]

\hline
\multirow{10}{1.9cm}
~ & 100 & ~~~~ $0.006270$ & ~~~~ $0.000043$&~  \\[1.2ex]
~&  120 & ~~~~ $0.005539$ & ~~~~ $0.000036$&~   \\[1.2ex]
0.3& 150 & ~~~~ $0.005343$ & ~~~~ $0.000032$&~~~~ $0.009078$  \\[1.2ex]
~ & 200 & ~~~~ $0.005303$ & ~~~~ $0.000030$&~   \\[1.2ex]

\hline
\multirow{10}{1.9cm}
~ & 100 & ~~~~ $0.005609$ & ~~~~ $0.000035$&~  \\[1.2ex]
~&  120 &~~~~  $0.005426$ & ~~~~ $0.000034$&~   \\[1.2ex]
0.4& 150 &~~~~  $0.005468$ &~~~~  $0.000033$&~~~~ $0.008114$  \\[1.2ex]
~ & 200 & ~~~~ $0.005173$ &~~~~  $0.000029$&~   \\[1.2ex]

\hline
\multirow{10}{1.9cm}
~ & 100 & ~~~~ $~~0.002632$ &~~~~  $0.000042$&~  \\[1.2ex]
~&  120 & ~~~~ $~~0.001417$ &~~~~  $0.000026$&~   \\[1.2ex]
1.0& 150 & ~~~~ $0.000612$ & ~~~~ $0.000019$&~~~~ $0.011937$  \\[1.2ex]
~ & 200 & ~~~~ $0.000387$ & ~~~~ $0.000016$&~   \\[1.2ex]
			\hline\hline	
			\label{tb1} 		
	\end{tabular}} 
	\end{table}

\subsection*{Real data set}
Here, we consider a real data set representing average daily wind speeds (in meter/second) in November, $2007$ at Elanora Heights, a northeastern suburb of Sydney, Australia. The real data set is presented in Table \ref{tb2} (see \cite{best2010easily}).
\begin{table}[ht!]
\caption {The data set.}
	\centering % used for centering table
	\scalebox{1}{\begin{tabular}{c c c c c c c c } % centered columns (4 columns)
		%	\hline\vspace{.1cm}
		\toprule
0.5833  ~~0.6667 ~~0.6944 ~~0.7222 ~~0.7500 ~~0.7778 ~~0.8056 ~~0.8056 ~~0.8611\\[.5Ex]
0.8889  ~~0.9167 ~~1.0000 ~~1.0278 ~~1.0278 ~~1.1111 ~~1.1111 ~~1.1111 ~~1.1667\\[.5Ex]
1.1667  ~~1.1944 ~~1.2778 ~~1.2778 ~~1.3056 ~~1.3333 ~~1.3333 ~~1.3611 ~~1.4444\\[.5Ex]
2.1111  ~~2.1389 ~~2.7778 \\
	%\hline	 
		\bottomrule
	\label{tb2} 			
	\end{tabular}} 
	\end{table}
For checking the best fitted model for the real data set, goodness of fit test has been applied. Here, we have considered four statistical models:  Gumbel-II (GMB-II), Weibull, generalised X-exponential (GXE),  and exponential (EXP) distributions.  We use negative log-likelihood criterion ($-\ln L$), Akaikes-information criterion (AIC), AICc, Bayesian information criterion (BIC) and the $p$-value related to Kolmogorov-Smirnov (K-S) test. From Table $3$, we observe that the GMB-II distribution fits better than other distributions as the values of the test statistics are smaller than that of the other distributions. The Gaussian kernel in (\ref{eq5.3*}) is employed as the kernel function for estimation purpose. The values of AB and MSE of the proposed estimate in (\ref{eq5.2}) have been computed using $500$ bootstrap samples with size $n=30$ and $b_n=0.35$. These are given in Table \ref{tb4} for different choices of $t$.

\begin{table}[ht!]
	\centering
	\caption {{The MLEs, BIC, AICc, AIC, and negative log-likelihood values of the statistical models for the data set presented in Table \ref{tb2}.}}
	\scalebox{1}{\begin{tabular}{cccccccc}
			\toprule
			%\multicolumn{1}{c}{} & \multicolumn{3}{c}{\textbf{Topic 2}} & \multicolumn{2}{c}{\textbf{Topic 3}} \\
			%\cmidrule(rl){2-4} \cmidrule(rl){5-6}
			\textbf{Model}  & \textbf{Shape}  & \textbf{Scale}  & \textbf{-ln L}  & \textbf{AIC}& \textbf{AICc} & \textbf{BIC} &\textbf{p-value} \\
			\midrule
			%\rowcolor{lavender}
			GMB-II  & $\widehat{\alpha}= 3.3869$ & $\widehat{\lambda}=0.7544$  &  12.5333 &  29.0665 &  29.5109 &  31.8689& 0.92340 \\[1.2ex]
			GXE &  $\widehat{\alpha}=4.1464$ & $\widehat{\lambda}=1.1421$ & 17.3245 & 38.6491&   39.0935 &  41.4515 & 0.67520 \\[1.2ex]
			Weibull  & $\widehat{\alpha}= 2.5393$ & $\widehat{\lambda}=1.3048$  &  18.5659 &  41.1317 &  41.5762&  43.9341 & 0.28650\\[1.2ex]
			EXP &  $\widehat{\lambda}=0.8633$ & ~  & 34.4095 &  70.8191 & 70.9620& 72.2203 & 0.00005\\[1ex]
		%	IEHL & $\widehat{\alpha}=11.6348$ & $\widehat{\lambda}=0.1582$  & 16.9492 & 37.8984 & 38.6043& 39.8899 \\
			%\rowcolor{lavender}
			%row5 & A.5 & B.5 & C.5 & D.5 & E.5 \\
			\bottomrule
			\label{tb3}
	\end{tabular}}
\end{table}

%Recently, \cite{maiti2023estimating} checked goodness of fit test for the purpose for best fitted, four statistical models were considered: They examined that the LL distribution was best fit for the data set in Table $2$ and they also evaluated the values of  unknown parameters: Shape parameter . We estimate the bias,  mean square error (MSE) of WPVE in (\ref{eq5.2}) and $\mathcal{\overline{VE}^\omega}(X;t)$ for different values of $t$ are given in Table $3$. The value of $b_n=0.35$ have been choosen for estimation purpose.

\begin{table}[ht!]
\caption {The AB, MSE  and  ${\it VE}^y(Y;t)$ for the  data set in Table $2$.}
	\centering % used for centering table
	\scalebox{1}{\begin{tabular}{c c c c c c c c } % centered columns (4 columns)
			\hline\hline\vspace{.1cm} %inserts double horizontal lines
			~$t$ &~~~~AB &~~~ MSE & ~~$\mathcal{\overline{VE}}^y(Y;t)$ \\
			\hline\hline
			
		% ~0.9 &~~~-0.12418 &~~~0.01695&~~~0.19397  \\[1ex]
		
			~1.0 &~~~0.05122 &~~~0.00453&~~~0.12205 \\[1ex]
			
			~1.1 &~~~0.02278 &~~~0.00135&~~~0.08995  \\[1ex]
			~1.2&~~~0.03233 &~~~0.00176 &~~~0.07566  \\[1ex]
			~1.3 &~~~0.05258 &~~~0.00314 &~~~0.08479    \\[1ex]
			~1.4&~~~0.09765 &~~~0.00959 &~~~0.12025  \\[1ex]	
			~1.5 &~~~0.14900 &~~~0.02237 &~~~0.18279   \\[1ex]
			~1.8 &~~~0.33773 &~~~0.11743 &~~~0.52167   \\[1ex]
			~2.0 &~~~0.45306 &~~~0.21668 &~~~0.85184   \\[1ex]
		
			~2.5&~~~0.59685 &~~~0.37733 &~~~1.98038   \\[1ex]
		%	~2.7&~~~0.14313 &~~~0.05297 &~~~1.89269   \\[1ex]
			~3.0 &~~~0.65917 & ~~~0.55132 &~~~3.10000   \\[1ex]

			\hline\hline
			\label{tb4} 		 		
	\end{tabular}} 
	
\end{table}

\subsubsection{Parametric estimation}
Here, we consider parametric estimation of the WPVE. We assume that the data are taken from an exponential population with parameter $\lambda.$ In this case, the WPVE is obtained as 
\begin{align}\label{eq5.4}
 \mathcal{\overline{VE}}^y(Y;t) = \int_{0}^{t}y^2\frac{\lambda e^{-\lambda y}}{ 1-e^{-\lambda t}}\bigg(\log\Big(\frac{\lambda e^{-\lambda y}}{ 1-e^{-\lambda t}}\Big)\bigg)^2dy-\int_{0}^{t}y\frac{\lambda e^{-\lambda y}}{ 1-e^{-\lambda t}}\log\Big(\frac{\lambda e^{-\lambda y}}{ 1-e^{-\lambda t}}\Big)dy.
 \end{align}
We apply maximum likelihood estimation technique for the purpose of estimation of (\ref{eq5.4}).  Let $\hat \lambda$ be the maximum likelihood estimate (MLE) of the model parameter $\lambda$. Using the invariance property, the MLE of the WPVE is obtained as 
\begin{align}\label{eq5.5}
 \mathcal{\widetilde{\overline{VE}}}^y(Y;t) = \int_{0}^{t}y^2\frac{\hat\lambda e^{-\hat{\lambda} y}}{ 1-e^{-\hat\lambda t}}\bigg(\log\Big(\frac{\hat\lambda e^{-\hat\lambda y}}{ 1-e^{-\hat\lambda t}}\Big)\bigg)^2dy-\int_{0}^{t}y\frac{\hat\lambda e^{-\hat\lambda y}}{ 1-e^{-\hat\lambda t}}\log\Big(\frac{\hat\lambda e^{-\hat\lambda y}}{ 1-e^{-\hat\lambda t}}\Big)dy,~t>0.
 \end{align}
Here, we conduct Monte-Carlo simulation using $R$ software to see the behaviour of the proposed parametric estimate of WPVE. We take $\lambda=0.7$ as its true parameter value. We have considered $t=0.1,0.2,0.3,0.4,1$ and $n=100,120, 150,200.$ Using $100$ replications, the AB and MSE have been computed and presented in Table \ref{tb5}. From Table \ref{tb5}, we observe that  MSEs decrease as $n$ increases.  

From Table \ref{tb1} and Table \ref{tb5}, we observe that the performance of the estimate of the WPVE using parametric approach is superior than the non-parametric approach in terms of the AB and MSE values.

\begin{table}[ht!]
\caption {The AB, MSE, and $\overline{\mathcal{VE}}^y(Y;t)$ for the parametric estimator of WPVE in (\ref{eq5.5}).}
	\centering % used for centering table
	\scalebox{1}{\begin{tabular}{c c c c c c c c } % centered columns (4 columns)
			\hline\hline\vspace{.1cm} %inserts double horizontal lines
			$t$ &$n$ & AB & MSE &$\overline{\mathcal{VE}}^y(Y;t)$ \\
			\hline\hline
\multirow{10}{1.9cm}
~ & 100 & $1.977\times 10^{-6}$ & $ 1.653\times 10^{-10}$&~   \\[1.2ex]
~&  120 & $1.961\times 10^{-6}$ & $1.240\times 10^{-10}$&~   \\[1.2ex]
0.1& 150 & $2.387\times 10^{-6}$ & $1.023\times 10^{-10}$&$0.004285$  \\[1.2ex]
~ & 200 & $1.181\times 10^{-6}$ & $6.874\times 10^{-11}$&~   \\[1.2ex]

\hline
\multirow{10}{1.9cm}
~ & 100 & $1.049\times 10^{-5}$ & $4.766\times 10^{-9}$&~  \\[1.2ex]
~&  120 & $1.044\times 10^{-5}$ & $3.578\times 10^{-9}$&~   \\[1.2ex]
0.2& 150 & $1.275\times 10^{-5}$ & $2.953\times 10^{-9}$&$0.007901$  \\[1.2ex]
~ & 200 & $ 6.293\times 10^{-6}$ & $1.987\times 10^{-9}$&~   \\[1.2ex]

\hline
\multirow{10}{1.9cm}
~ & 100 & $2.421\times 10^{-5}$ & $2.647\times 10^{-8}$&~  \\[1.2ex]
~&  120 & $2.424\times 10^{-5}$ & $1.989\times 10^{-8}$&~   \\[1.2ex]
0.3& 150 & $2.977\times 10^{-5}$ & $1.641\times 10^{-8}$&$0.009078$  \\[1.2ex]
~ & 200 & $1.463\times 10^{-5}$ & $1.108\times 10^{-8}$&~   \\[1.2ex]

\hline
\multirow{10}{1.9cm}
~ & 100 & $3.765\times 10^{-5}$ & $6.900\times 10^{-8}$&~  \\[1.2ex]
~&  120 & $3.809\times 10^{-5}$ & $ 5.193\times 10^{-8}$&~   \\[1.2ex]
0.4& 150 & $4.728\times 10^{-5}$ & $4.284\times 10^{-8}$&$0.008114$  \\[1.2ex]
~ & 200 & $2.305\times 10^{-5}$ & $2.908\times 10^{-8}$&~   \\[1.2ex]

\hline
\multirow{10}{1.9cm}
~ & 100 & $4.220\times 10^{-4}$ & $5.061\times 10^{-6}$&~  \\[1.2ex]
~&  120 & $3.981\times 10^{-4}$ & $ 3.758\times 10^{-6}$&~   \\[1.2ex]
1.0& 150 & $4.595\times 10^{-4}$ & $ 3.105\times 10^{-6}$&$0.011937$  \\[1.2ex]
~ & 200 & $2.369\times 10^{-4}$ & $2.015\times 10^{-6}$&~   \\[1.2ex]
\hline\hline
\label{tb5} 		 		
	\end{tabular}} 
	\end{table}

\subsection{WPDVE}
In this subsection, we have proposed non-parametric and parametric estimates of the WPDVE. Below, we discuss the non-parametric estimate.
\subsubsection{Non-parametric estimation}
Similar to (\ref{eq5.2}), a kernel-based non-parametric estimate of the WPDVE in (\ref{eq4.3}) is obtained as
\begin{align}\label{eq5.6}
\widehat{\mathcal{PVE}}^y(Y;t)&=\int_{0}^{t}\widehat{\eta}(y)(y\log  (\widehat{\eta}(y)))^2dy+\int_{t}^{\infty}\widehat{\delta}(y)(y\log  (\widehat{\delta}(y)))^2dy-\left[\int_{0}^{t}y\widehat{\eta}(y)\log  (\widehat{\eta}(y))dy\right]^2\nonumber\\
&~~-\left[\int_{t}^{\infty}y\widehat{\delta}(y)\log  (\widehat{\delta}(y))dy\right]^2,
\end{align}
where $\widehat{\delta}(y)=\frac{\widehat{g}(y)}{\widehat{\bar G}(t)}$, $\widehat{\eta}(y)=\frac{\widehat g(y)}{\widehat G(t)}$,  $\widehat{\bar G}(t)=\int_{t}^{\infty}\widehat{g}(y)dy$, and $\widehat G(t)=\int_{0}^{t}\widehat g(y)dy$. 

\subsection*{Simulation study}
Similar to the preceding subsection, here a Monte-Carlo simulation study has been carried out to check the performance of the proposed kernel-based non-parametric estimate of the WPDVE given in (\ref{eq5.6}). The data set has been generated from exponential distribution with $\lambda=5$ using ``Mathematica" software. For different values of $n=100,120,150,200$ and $t=0.05,0.1,0.15,0.2$, the AB and MSE values have been computed using $500$ replications.  We have used the Gaussian kernel given in (\ref{eq5.3*}). The computed values of the AB and MSE are presented in Table  \ref{tb6}. From Table  \ref{tb6}, we notice similar observation to the case of WPVE. 

\begin{table}[ht!]
\caption {The AB, MSE, and  $\mathcal{PVE}^y(Y;t)$ for the kernel estimator of  WPDVE in (\ref{eq5.6}).}
	\centering % used for centering table
	\scalebox{1}{\begin{tabular}{c c c c c c c c } % centered columns (4 columns)
			\hline\hline\vspace{.1cm} %inserts double horizontal lines
			$t$ &$n$ & AB &MSE &$\mathcal{PVE}^y(Y;t)$ \\
			\hline\hline
\multirow{10}{1.9cm}
~ & 100 & $0.15322$ & $0.04596$&~   \\[1.2ex]
~&  120 & $0.15714$ & $0.04499$&~   \\[1.2ex]
0.05& 150 & $0.13113$ & $0.03731$&$0.44062$  \\[1.2ex]
~ & 200 & $0.11894$ & $0.03376$&~   \\[1.2ex]

\hline
\multirow{10}{1.9cm}
~ & 100 & $0.19768$ & $0.06524$&~  \\[1.2ex]
~&  120 & $0.187725$ & $0.06388$&~   \\[1.2ex]
0.10& 150 & $0.16648$ & $0.05171$& $0.49772$  \\[1.2ex]
~ & 200 & $0.14099$ & $0.04369$&~   \\[1.2ex]

\hline
\multirow{10}{1.9cm}
~ & 100 & $0.23206$ & $0.08678$&~  \\[1.2ex]
~&  120 & $0.20336$ & $0.07177$&~   \\[1.2ex]
0.15& 150 & $0.20001$ & $0.06897$&$0.55890$  \\[1.2ex]
~ & 200 & $0.18718$ & $0.06399$&~   \\[1.2ex]

\hline
\multirow{10}{1.9cm}
~ & 100 & $0.26717$ & $0.11190$&~  \\[1.2ex]
~&  120 & $0.26187$ & $0.10100$&~   \\[1.2ex]
0.20& 150 & $0.23457$ & $0.09309$ & $0.62414$  \\[1.2ex]
~ & 200 & $0.22166$ & $0.07981$&~   \\[1.2ex]
\hline\hline
\label{tb6} 		 		
	\end{tabular}} 
	\end{table}

\subsubsection{Parametric estimation}     
Consider an exponential population with mean $1/\lambda,$ $\lambda>0.$ The WPDVE of the exponential distribution is 
  \begin{align}\label{eq5.7}
  \mathcal{PVE}^y(Y;t)&=\int_{0}^{t}y^2\frac{\lambda e^{-\lambda y}}{ 1-e^{-\lambda t}}\bigg(\log\Big(\frac{\lambda e^{-\lambda y}}{ 1-e^{-\lambda t}}\Big)\bigg)^2dy-\int_{0}^{t}y\frac{\lambda e^{-\lambda y}}{ 1-e^{-\lambda t}}\log\Big(\frac{\lambda e^{-\lambda y}}{ 1-e^{-\lambda t}}\Big)dy\nonumber\\
  &+\int_{t}^{\infty}\lambda e^{(t-y)\lambda}\Big(y\log(\lambda e^{(t-y)\lambda})\Big)^2dy-\int_{t}^{\infty}\lambda y e^{(t-y)\lambda}\log(\lambda e^{(t-y)\lambda})dy.
  \end{align}  
 Firstly, we  estimate the model parameter $\lambda$ using maximum likelihood estimation technique  for estimating (\ref{eq5.7}). The MLE of $\mathcal{PVE}^y(Y;t)$ in (\ref{eq5.7}) is
 \begin{align}\label{eq5.8}
 \widetilde{\mathcal{PVE}}^y(Y;t)=&=\int_{0}^{t}y^2\frac{\hat\lambda e^{-\hat\lambda y}}{ 1-e^{-\hat\lambda t}}\bigg(\log\Big(\frac{\hat\lambda e^{-\hat\lambda y}}{ 1-e^{-\hat\lambda t}}\Big)\bigg)^2dy-\int_{0}^{t}y\frac{\hat\lambda e^{-\hat\lambda y}}{ 1-e^{-\hat\lambda t}}\log\Big(\frac{\hat\lambda e^{-\hat\lambda y}}{ 1-e^{-\hat\lambda t}}\Big)dy\nonumber\\
   &+\int_{t}^{\infty}\hat\lambda e^{(t-y)\hat\lambda}\Big(y\log(\hat\lambda e^{(t-y)\hat\lambda})\Big)^2dy-\int_{t}^{\infty}\hat\lambda y e^{(t-y)\hat\lambda}\log(\hat\lambda e^{(t-y)\hat\lambda})dy,
 \end{align}
 where $\hat \lambda$ is the MLE of $\lambda$.  To evaluate the performance of the proposed parametric estimate,  Monte-Carlo simulation is conducted using $R$ software with $500$ replications. Here, we consider the true value of $\lambda$ as $5$. For sample sizes $n=100, 120, 150$ and $200$, the AB and MSE values have been presented in Table \ref{tb7} for different choices of $t$ . From Table \ref{tb7}, we observe that  MSEs decrease as $n$ increases, which assures the consistency and validation of the propose estimate $ \widetilde{\mathcal{PVE}}^y(Y;t)$ in (\ref{eq5.8}).
 
  From Table \ref{tb6} and Table \ref{tb7}, we observe that the parametric estimate in (\ref{eq5.8}) performs better than the non-parametric estimate in (\ref{eq5.6}) when the data are generated from  exponential distribution with $\lambda=5$ in terms of the AB and MSE.

 \begin{table}[ht!]
 \caption {The AB, MSE, and $\mathcal{PVE}^y(Y;t)$ for parametric  estimate of  WPDVE in (\ref{eq5.8}).}
 	\centering % used for centering table
 	\scalebox{1}{\begin{tabular}{c c c c c c c c } % centered columns (4 columns)
 			\hline\hline\vspace{.1cm} %inserts double horizontal lines
 			$t$ &$n$ & AB &MSE &$\mathcal{PVE}^y(Y;t)$ \\
 			\hline\hline
 \multirow{10}{1.9cm}
 ~ & 100 & $ 0.00884$ & $0.01091$&~   \\[1.2ex]
 ~&  120 & $ 0.00923$ & $0.00927$&~   \\[1.2ex]
 0.05& 150 & $0.00814$ & $0.00712$&$0.44062$  \\[1.2ex]
 ~ & 200 & $0.00618$ & $0.00526$&~   \\[1.2ex]
 
 \hline
 \multirow{10}{1.9cm}
 ~ & 100 & $0.00905$ & $0.01236$&~  \\[1.2ex]
 ~&  120 & $0.00953$ & $0.01049$&~   \\[1.2ex]
 0.10& 150 & $0.00843$ & $0.00807$& $0.49772$  \\[1.2ex]
 ~ & 200 & $0.00641$ & $0.00597$&~   \\[1.2ex]
 
 \hline
 \multirow{10}{1.9cm}
 ~ & 100 & $ 0.00927$ & $0.01390$&~  \\[1.2ex]
 ~&  120 & $0.00982$ & $0.01181$&~   \\[1.2ex]
 0.15& 150 & $0.00872$ & $0.00909$&$0.55890$  \\[1.2ex]
 ~ & 200 & $0.00663$ & $0.00672$&~   \\[1.2ex]
 
 \hline
 \multirow{10}{1.9cm}
 ~ & 100 & $0.00948$ & $0.01554$&~  \\[1.2ex]
 ~&  120 & $0.01012$ & $0.01319$&~   \\[1.2ex]
 0.20& 150 & $0.00901$ & $0.01017$ & $0.62414$  \\[1.2ex]
 ~ & 200 & $0.00686$ & $0.00752$&~   \\[1.2ex]
 \hline\hline
 \label{tb7} 		 		
 	\end{tabular}} 
 	\end{table}

 \section{Application in reliability engineering} \label{sec6}
  In reliability engineering, a coherent system is a model used to analyse the performance and reliability of systems composed of multiple components. The key idea is to understand how the configuration and interdependence of components affect the overall system reliability. This allows engineers to analyse how the failure or success of components impacts the entire system. For instance, in a series system, failure of any single component leads to the failure of the whole system, while in a parallel system, the system continues to operate as long as at least one component is functioning.
  
 We consider a coherent system with $n$ components and lifetime of the coherent system is
 denoted by $T$. The random lifetimes of $n$ components of the coherent system are identically distributed (i.d.) with a common CDF and PDF $G(\cdot)$ and $g(\cdot)$, respectively. The CDF and PDF of the coherent system with lifetime $T$ are defined as
 \begin{align}
 G_T(y)=q(G(x)) ~~~ \text{and}~~~ g_T(y)=q^{\prime}(G(x))g(y),
 \end{align}
respectively, where $q:[0,1]\rightarrow[0,1]$ is a distortion function (see \cite{navarro2013stochastic}) and $q^{\prime}\equiv\frac{dq}{dy}$.  The distortion function which is increasing and continuous function,  depends on the structure of a system and the copula of the component lifetimes and $q(0)=0,~q(1)=1$. Several researchers discussed the coherent system for different information measures as an application, one may refer to \cite{toomaj2017some}, \cite{cali2020properties} and \cite{saha2024weighted}. The WPVE of $T$ is defined by
\begin{align}\label{eq6.2}
\overline{\mathcal{VE}}^y(T)&=\int_{0}^{t}\phi(G_T(y))dy-\bigg(\int_{0}^{t}\psi(G_T(y))dy\bigg)^2\nonumber\\
&=\int_{0}^{G(t)}\frac{\phi\big(q(u)\big)}{g\big(G^{-1}(u)\big)}du-\bigg(\int_{0}^{G(t)}\frac{\psi\big(q(u)\big)}{g\big(G^{-1}(u)\big)}du\bigg)^2,~~~u=G(y)
\end{align}
where $$\phi\big(q(u)\big)=\frac{g_T(G_T^{-1}(q(u))}{G_T(t)}\Big[G_T^{-1}(u)\log \Big(\frac{g_T(G_T^{-1}(q(u)))}{G_T(t)}\Big)\Big]^2$$ and $$\psi\big(q(u)\big)=G_T^{-1}(u)\frac{g_T(G_T^{-1}(q(u)))}{G_T(t)}\log \Big(\frac{g_T(G_T^{-1}(q(u)))}{G_T(t)}\Big).$$  
  
Next, we explore an example of the WPVE of a coherent system for illustration purpose.
\begin{example}\label{ex6.1}
Suppose $Y_1$, $Y_2$ and $Y_3$ denote the lifetimes of the components of a coherent
system. Assume that they all follow power distribution with CDF $G(y)=y^\beta,~y\in[0,1]$ and $\beta>0$.
 We consider a parallel system with lifetime $T = X_{3:3} = \max\{Y_1, Y_2, Y_3\}$ whose distortion
function is $q(v) = v^3, ~0\le v \le 1$. Thus, from (\ref{eq6.2}), the WPVE of the coherent system is obtained as
\begin{align}\label{eq6.3}
\overline{\mathcal{VE}}^y(T)&=\frac{3\beta t^2}{(2+3\beta)^3}\bigg[\Big\{(3\beta+2)\log\Big(\frac{3\beta}{t^{3\beta}}\Big)-3\beta+1\Big\}^2+(3\beta-1)^2\Big\{(3\beta+2)\log(t)-1\Big\}^2\nonumber\\
&+2(3\beta+2)^2(3\beta-1)\log\Big(\frac{3\beta}{t^{3\beta}}\Big)\log(t)\bigg]-\frac{81\beta^2t^\frac{2}{3}}{(9\beta+1)^4}\bigg\{(9\beta+1)\log\Big(\frac{3\beta}{t^{3\beta}}\Big)\nonumber\\
&+3(3\beta-1)\Big(\log(t^{3\beta+\frac{1}{3}})-1\Big)\bigg\}^2. 
\end{align}
The graphical presentation of WPVE for parallel system in (\ref{eq6.3}) is given in Figure \ref{fig7} with respect to $t$ (when $\beta$ is fixed) and $\beta$ (when $t$ is fixed).
\end{example} 

\begin{figure}[h!]
	\centering
	\subfigure[]{\label{c1}\includegraphics[height=2in]{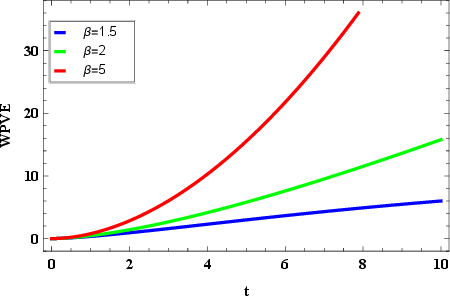}}
	\subfigure[]{\label{c1}\includegraphics[height=2in]{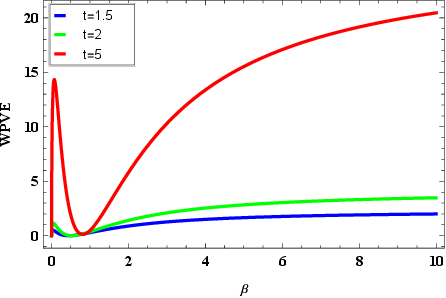}}
	\caption{ Plots of the WPVE  for parallel system with three system in Example \ref{ex6.1} $(a)$  with respect to $t$ for $\beta=1.5,2$ and $5$ and $(b)$ with respect to $\beta$ for $t=1.5,2$ and $5$.}
	\label{fig7}	
\end{figure}

Now, the relation between the WPVE of coherent system and component has been established in the following result.
\begin{proposition}\label{prop6.1}
Suppose $T$ is the lifetime of a coherent system with identically distributed components. The component lifetime is $Y$ with CDF $G(\cdot)$ and PDF $g(\cdot)$ and the CDF and PDF of $T$ are $G_T(\cdot)$ and $g_T(\cdot)$, respectively, and $q(\cdot)$ denotes the distortion function. Assume that $$\phi\big(q(u)\big)=\frac{g_T(G_T^{-1}(q(u))}{G_T(t)}\Big[G_T^{-1}(u)\log \Big(\frac{g_T(G_T^{-1}(q(u)))}{G_T(t)}\Big)\Big]^2$$ and $$\psi\big(q(u)\big)=G_T^{-1}(u)\frac{g_T(G_T^{-1}(q(u)))}{G_T(t)}\log \Big(\frac{g_T(G_T^{-1}(q(u)))}{G_T(t)}\Big),$$ for all $0\le u\le G(t).$ If $\phi(q(u))\ge (\le) \phi(u)$ and  $\psi(q(u))\le (\ge) \psi(u)$, for all $0\le u\le G(t),~t>0$, then
\begin{align}
\overline{\mathcal{VE}}^y(T)\ge (\le)\overline{\mathcal{VE}}^y(Y).
\end{align} 
\end{proposition}  
\begin{proof}
Take $\phi(q(u))\ge  \phi(u)$ and  $\psi(q(u))\le \psi(u)$, for all $0\le u\le G(t),~t>0$. Then, we obtain
\begin{align}\label{eq6.5}
\int_{0}^{G(t)}\frac{\phi(q(u))}{g(G^{-1}(u))}du\ge \int_{0}^{G(t)}\frac{\phi(u)}{g(G^{-1}(u))}du
\end{align}
and 
\begin{align}\label{eq6.6}
\int_{0}^{G(t)}\frac{\psi(q(u))}{g(G^{-1}(u))}du\le \int_{0}^{G(t)}\frac{\psi(u)}{g(G^{-1}(u))}du.
\end{align}
Using (\ref{eq6.5}) and (\ref{eq6.6}), we easily obtain that $\overline{\mathcal{VE}}^y(T)\ge \overline{\mathcal{VE}}^y(Y).$ The other part of the proof is similar, therefore omitted for the brevity. Hence, the result is made.
\end{proof}  
  
 The upper bound of the WPVE of coherent system in terms of the weighted past SE and CRHR is established. 
\begin{proposition}
Consider a coherent system as in Proposition \ref{prop6.1}. Denote $\sup_{u\in[0,G(t)]}\frac{\phi(q(u))}{\phi(u)}=\eta_{1,u}$, where  $$\phi\big(q(u)\big)=\frac{g_T(G_T^{-1}(q(u))}{G_T(t)}\Big[G_T^{-1}(u)\log \Big(\frac{g_T(G_T^{-1}(q(u)))}{G_T(t)}\Big)\Big]^2$$ and $$\psi\big(q(u)\big)=G_T^{-1}(u)\frac{g_T(G_T^{-1}(q(u)))}{G_T(t)}\log \Big(\frac{g_T(G_T^{-1}(q(u)))}{G_T(t)}\Big).$$ Then, under the condition in (\ref{eq2.3}), we obtain
\begin{align}
\overline{\mathcal{VE}}^y(T)\le\frac{\eta_{1,u}}{G(t)}\overline{\mathcal{H}}^{\omega_2}(Y;t)+\big(\Lambda^{*}(t)\big)^2E[Y^2|Y\le t],
\end{align}
where $\overline{\mathcal{H}}^{\omega_2}(Y;t)$ is weighted past SE with weight $\omega_2(y)=\alpha y+\beta y^2$.
\end{proposition}  
  
\begin{proof}
From (\ref{eq6.2}) and using (\ref{eq2.3}), we obtain
\begin{align}\label{eq6.8}
\overline{\mathcal{VE}}^y(T)
&\le\int_{0}^{G(t)}\frac{\phi\big(q(u)\big)}{g\big(G^{-1}(u)\big)}du\nonumber\\
&\le\Big(\sup_{u\in[0,G(t)]}\frac{\phi(q(u))}{\phi(u)}\Big)\int_{0}^{G(t)}\frac{\phi(u)}{g\big(G^{-1}(u)\big)}du\nonumber\\
&=\eta_{1,u}\int_{0}^{t}y^2\frac{g(y)}{G(t)}\Big(\log\Big(\frac{g(y)}{G(t)}\Big)\Big)^2dy\\
&\le\frac{\eta_{1,u}}{G(t)}\bigg[\int_{0}^{t}y^2g(y)\Big(\log\big(g(y)\big)\Big)^2dy+\Big(\log\big(G(t)\big)\Big)^2\int_{0}^{t}y^2g(y)dy\bigg]\nonumber\\
&\le\frac{\eta_{1,u}}{G(t)}\bigg[\int_{0}^{t}y^2(-\alpha y-\beta)g(y)\log\big(g(y)\big)dy+\Big(\log\big(G(t)\big)\Big)^2\int_{0}^{t}y^2g(y)dy\bigg]\nonumber\\
&=\frac{\eta_{1,u}}{G(t)}\overline{\mathcal{H}}^{\omega_2}(Y;t)+\big(\Lambda^{*}(t)\big)^2E[Y^2|Y\le t].
\end{align}
Therefore, the proof is completed.
\end{proof} 
 Next, we obtain an upper bound of the WPVE of coherent system in terms of WPVE and weighted past SE of the component.  
\begin{proposition}
Consider a coherent system as in Proposition \ref{prop6.1}. Denote $\sup_{u\in[0,G(t)]}\frac{\phi(q(u))}{\phi(u)}=\eta_{1,u}$, where  $$\phi\big(q(u)\big)=\frac{g_T(G_T^{-1}(q(u))}{G_T(t)}\Big[G_T^{-1}(u)\log \Big(\frac{g_T(G_T^{-1}(q(u)))}{G_T(t)}\Big)\Big]^2$$ and $$\psi\big(q(u)\big)=G_T^{-1}(u)\frac{g_T(G_T^{-1}(q(u)))}{G_T(t)}\log \Big(\frac{g_T(G_T^{-1}(q(u)))}{G_T(t)}\Big).$$ Then,
\begin{align*}
\overline{\mathcal{VE}}^y(T)\le\eta_{1,u}\Big\{\overline{\mathcal{VE}}^{y}(Y;t)+\big(\overline{\mathcal{H}}^{y}(Y;t)\big)^2\Big\}.
\end{align*}
\end{proposition}  
\begin{proof}
The proof  follows directly from (\ref{eq6.8}). Hence, we omit the proof for brevity. 
\end{proof}  
  
\begin{proposition}
Consider a coherent system in Proposition \ref{prop6.1}. Assume that the components have PDF $g(y)$ with support $S$, such that $g(y)\ge L>0 ~\forall~y\in S.$ Then, we obtain
\begin{align*}
\overline{\mathcal{VE}}^y(T)\le\frac{1}{L}\int_{0}^{G(t)}\phi\big(q(u)\big)du,
\end{align*}
where $\phi\big(q(u)\big)=\frac{g_T(G_T^{-1}(q(u))}{G_T(t)}\Big[G_T^{-1}(u)\log \Big(\frac{g_T(G_T^{-1}(q(u)))}{G _T(t)}\Big)\Big]^2$.
\end{proposition}  
 \begin{proof}
 From (\ref{eq2.1}), we have
 \begin{align*}
 \overline{\mathcal{VE}}^y(T)\le\int_{0}^{G(t)}\frac{\phi\big(q(u)\big)}{g\big(G^{-1}(u)\big)}du\le\frac{1}{L}\int_{0}^{G(t)}\phi\big(q(u)\big)du.
 \end{align*}
 Therefore, the result is made.
 \end{proof}

  Next, a comparative study is carried out between the proposed WPVE, past VE (due to \cite{buono2022varentropy}), weighted past R\'enyi entropy (due to \cite{nourbakhsh2017weighted}) and weight past SE  (due to \cite{di2007weighted}) for three different coherent systems with three components. Suppose  $T$ and $Y$ denote the  system's lifetime and component's lifetime  with PDFs $g_T(\cdot)$ and $g(\cdot)$ and CDFs $G_T(\cdot)$ and $G(\cdot)$, respectively. The weighted past SE and weighted past R\'enyi entropy of $T$ are 
  \begin{align}\label{eq6.10}
  \overline{\mathcal{H}}^y(T)=-\int_{0}^{G(t)}\frac{\psi\big(q(u)\big)}{g\big(G^{-1}(u)\big)}du,
   \end{align}
  and
  \begin{align}\label{eq6.11}
 \overline{\mathcal{H}}^y_\alpha(T)=\frac{1}{1-\alpha}\log\int_{0}^{1}\frac{\xi\big(q(u)\big)}{g\big(G^{-1}(u)\big)}du,~~\alpha>0~(\ne1),
  \end{align}
  respectively.  Further, the past VE of $T$ is 
  \begin{align}\label{eq6.12}
  \overline{\mathcal{VE}}(T)
  &=\int_{0}^{G(t)}\frac{\phi\big(q(u)\big)}{g\big(G^{-1}(u)\big)}du-\bigg(\int_{0}^{G(t)}\frac{\psi\big(q(u)\big)}{g\big(G^{-1}(u)\big)}du\bigg)^2,
  \end{align}
  where $\phi\big(q(u)\big)=\frac{g_T(G_T^{-1}(q(u))}{G_T(t)}\Big[G_T^{-1}(u)\log \Big(\frac{g_T(G_T^{-1}(q(u)))}{G_T(t)}\Big)\Big]^2$, $\xi\big(q(u)\big)=\Big(G_T^{-1}(u)\frac{g_T(G_T^{-1}(q(u))}{G_T(t)}\Big)^\alpha$ and $\psi\big(q(u)\big)=G_T^{-1}(u)\frac{g_T(G_T^{-1}(q(u)))}{G_T(t)}\log \Big(\frac{g_T(G_T^{-1}(q(u)))}{G_T(t)}\Big)$.   Here, we consider the power distribution with CDF $G(y)=x^\beta,~x>0,~\beta>0$, as a baseline distribution (component lifetime)  for illustrative purpose. We take three coherent systems: series system ($X_{1:3}$), 2-out-of-3 system ($X_{2:3}$), and parallel system ($X_{3:3}$) for evaluating the values of $\overline{\mathcal{VE}}^y(T)$ in (\ref{eq6.2}), $\overline{\mathcal{VE}}(T)$ in (\ref{eq6.12}), $\overline{\mathcal{H}}^y_\alpha(T)$ in (\ref{eq6.11}), and $\overline{\mathcal{H}}^y(T)$ in (\ref{eq6.10}). The numerical values of the WPVE, past VE, weighted past R\'enyi entropy, and weighted past SE for the series, 2-out-of-3, and parallel systems with $\alpha=1.8$, $\beta=0.2$ and $t=0.5$ are presented in Table \ref{tb4*}. As expected, from Table \ref{tb4*}, we observe that the uncertainty values of the series system are maximum; and minimum for parallel system considering all information measures, validating the proposed WPVE.

  \begin{table}[ht!]
  	\centering
  	\caption {The values of the WPVE, past varentropy (PVE), weighted past R\'enyi entropy (WPRE), and weighted past Shannon entropy (WPSE) for the series, 2-out-of-3, and parallel systems.}
  	\scalebox{1}{\begin{tabular}{ccccccc}
  			\toprule
  			%\multicolumn{1}{c}{} & \multicolumn{3}{c}{\textbf{Topic 2}} & \multicolumn{2}{c}{\textbf{Topic 3}} \\
  			%\cmidrule(rl){2-4} \cmidrule(rl){5-6}
  			\textbf{System}  & \textbf{WPVE}  & \textbf{PVE}  & \textbf{WPRE}  & \textbf{WPSE}  \\
  			\midrule
  			%\rowcolor{lavender}
  			Series ($X_{1:3}$)  & $0.016617$ & $26.558290$  &  $8.212988$ &  $0.015058$ \\[1.2ex]
  		2-out-of-3 ($X_{2:3}$) & $0.014338$ & $4.761798$  & $4.942339$&  $0.009428$ \\[1.2ex]
  		   	Parallel ($X_{3:3}$) & $0.001315$ & $0.444444$  & $2.931252$ &  $-0.081060$ \\[1.2ex]
  		   
  			\bottomrule
  			\label{tb4*}
  	\end{tabular}}
  \end{table}

\section{Conclusions}\label{sec7}

In this work, we have introduced WPVE and discussed its various properties. Bounds of the WPVE have been obtained. Sometimes it is very tough to obtain explicit expression of the WPVE of a transformed RV. To overcome such difficulties, in this paper, we have proposed a theorem, dealing with strictly monotone transformations. We have also introduced WPVE for PRHR model and explore some properties. Several examples have been considered for the purpose of illustration of the established theoretical results. Further, we proposed WPDE and WPDVE, and studied their several properties. It is observed that the WPDVE is a generalisation of the weighted varentropy and varentropy. The effectiveness of the WPDVE under affine transformations has been investigated. Lower and upper bounds of the WPDVE are derived. Furthermore, kernel-based non-parametric estimates for the WPVE and WPDVE have been proposed.  A simulation study is caried out to see the performance of the proposed non-parametric estimates. In order to compare the non-parametric estimation method with the parametric estimation method, we have considered parametric estimation of both WPDE and WPDVE. It is noticed that the parametric estimation method provides a better result than the non-parametric estimation method in the terms of the AB and MSE values when the data are generated from exponential distribution. A real data set representing the average wind speed has been considered and analysed for the purpose of estimation of the WPVE. Finally, an application of the WPVE in reliability engineering has been provided.

\section*{Acknowledgements} The authors thank the Associate editor and  referees for all their helpful comments and suggestions, which led to the substantial improvements. 
	Shital Saha thanks the University Grants Commission (Award No. 191620139416), India, for financial assistantship to carry out this research work. The two authors thank the research facilities provided by the Department of Mathematics, National Institute of Technology Rourkela, India.

%\section*{Declaration of competing interest} The authors declare that they have no known competing financial interests or personal relationships that could have appeared
%to influence the work reported in this paper.

\section*{Abbreviations}
				\begin{acronym}
					\acro{RV:}{Random variable} 
				    \acro{PDF:}{Probability density function} 
				    \acro{IC:}{Information content}
					\acro{SE:}{Shannon entropy}
			    	\acro{VE:}{Varentropy}
			    	\acro{MSE:}{Mean squared error}
			    	\acro{RVE:}{Residual varentropy}
					\acro{CDF:}{Cumulative distribution function}
					\acro{PVE:}{Past varentropy}
					\acro{WVE:}{Weighted varentropy} 
					\acro{WPSE:}{Weighted past Shannon entropy}
					\acro{WRVE:}{Weighted residual varentropy}
					\acro{WPVE:}{Weighted past varentropy}
					\acro{WPRE:}{Weighted past R\'enyi entropy}
					\acro{PRHR:}{Proportional reversed hazard rate}
					\acro{WPDE:}{Weighted paired dynamic entropy}
					\acro{WPDVE:}{Weighted paired dynamic varentropy entropy}
					\acro{AB:}{Absolute  bias}
					\acro{CRHR:}{Cumulative reversed hazard rate}
			    	\acro{VPL:}{Variance past lifetime}
			    	\acro{MPL:}{Mean past lifetime}
				    \acro{MRL:}{Mean residual lifetime}
				    \acro{VRL:}{Variance residual lifetime}	
				    \acro{GMB-II:}{Gumbel-II}
				    \acro{GXE:}{Generalised X-exponential}
				    \acro{EXP:}{Exponential}
				    \acro{lnL:}{Log-likelihood criterion}
				    \acro{AIC:}{Akaikes information criterion}
				    \acro{BIC:}{Bayesian information criterion}
				    \acro{MLE:}{Maximum likelihood estimate}	
\end{acronym}

\bibliography{refference}

@article{saha2024different,
  title={Different informational characteristics of cubic transmuted distributions},
  author={Saha, Shital and Kayal, Suchandan and Balakrishnan, N},
  journal={Brazilian Journal of Probability and Statistics},
  volume={38},
  number={2},
  pages={193--214},
  year={2024},
  publisher={Brazilian Statistical Association}
}

@article{nourbakhsh2017weighted,
  title={Weighted Renyi's entropy for lifetime distributions},
  author={Nourbakhsh, Mohammadreza and Yari, Gholamhoseein},
  journal={Communications in Statistics-Theory and Methods},
  volume={46},
  number={14},
  pages={7085--7098},
  year={2017},
  publisher={Taylor \& Francis}
}

@article{toomaj2017some,
  title={Some properties of the cumulative residual entropy of coherent and mixed systems},
  author={Toomaj, Abdolsaeed and Sunoj, SM and Navarro, Jorge},
  journal={Journal of Applied Probability},
  volume={54},
  number={2},
  pages={379--393},
  year={2017},
  publisher={Cambridge University Press}
}

@article{sharma2024stochastic,
  title={On stochastic properties of past varentropy with applications},
  author={Sharma, Akash and Kundu, Chanchal},
  journal={Applications of Mathematics},
  volume={69},
  number={3},
  pages={373--394},
  year={2024},
  publisher={Springer}
}

@article{hammer2000inequalities,
  title={Inequalities for Shannon entropy and Kolmogorov complexity},
  author={Hammer, Daniel and Romashchenko, Andrei and Shen, Alexander and Vereshchagin, Nikolai},
  journal={Journal of Computer and System Sciences},
  volume={60},
  number={2},
  pages={442--464},
  year={2000},
  publisher={Elsevier}
}

@article{mahdy2016further,
	title={Further results related to variance past lifetime class \& associated orderings and their properties},
	author={Mahdy, Mervat},
	journal={Physica A: Statistical Mechanics and its Applications},
	volume={462},
	pages={1148--1160},
	year={2016},
	publisher={Elsevier}
}

@article{alizadeh2023varentropy,
  title={Varentropy estimators with applications in testing uniformity},
  author={Alizadeh Noughabi, Hadi and Shafaei Noughabi, Mohammad},
  journal={Journal of Statistical Computation and Simulation},
  volume={93},
  number={15},
  pages={2582--2599},
  year={2023},
  publisher={Taylor \& Francis}
}

@article{saha2024weighted,
  title={Weighted (residual) varentropy and its applications},
  author={Saha, Shital and Kayal, Suchandan},
  journal={Journal of Computational and Applied Mathematics},
  volume={442},
  pages={115710},
  year={2024},
  publisher={Elsevier}
}

@article{sharma2023varentropy,
	title={Varentropy of doubly truncated random variable},
	author={Sharma, Akash and Kundu, Chanchal},
	journal={Probability in the Engineering and Informational Sciences},
	volume={37},
	number={3},
	pages={852--871},
	year={2023},
	publisher={Cambridge University Press}
}

@article{best2010easily,
  title={Easily applied tests of fit for the {R}ayleigh distribution},
  author={Best, D John and Rayner, John CW and Thas, Olivier},
  journal={Sankhya B},
  volume={72},
  pages={254--263},
  year={2010},
  publisher={Springer}
}

@book{cover1991elements,
  title={Elements of information theory},
  author={Cover, T.M. and Thomas, J.M},
  year={1991},
  publisher={John Wiley \& Sons}
}

@article{saha2023extended,
  title={Extended fractional cumulative past and paired $\phi$-entropy measures},
  author={Saha, Shital and Kayal, Suchandan},
  journal={Physica A: Statistical Mechanics and its Applications},
  volume={614},
  pages={128552},
  year={2023},
  publisher={Elsevier}
}

@article{li2008mixture,
  title={A mixture model of proportional reversed hazard rate},
  author={Li, Xiaohu and Li, Zhouping},
  journal={Communications in Statistics—Theory and Methods},
  volume={37},
  number={18},
  pages={2953--2963},
  year={2008},
  publisher={Taylor \& Francis}
}

@article{finkelstein2002reversed,
  title={On the reversed hazard rate},
  author={Finkelstein, Maxim S},
  journal={Reliability Engineering \& System Safety},
  volume={78},
  number={1},
  pages={71--75},
  year={2002},
  publisher={Elsevier}
}

@article{popovic2021generalized,
  title={Generalized proportional reversed hazard rate distributions with application in medicine},
  author={Popovi{\'c}, Bo{\v{z}}idar V and Gen{\c{c}}, Ali {\.I} and Domma, Filippo},
  journal={Statistical Methods \& Applications},
  volume={31},
  pages={459--480 },
  year={2022},
  publisher={Springer}
}

@article{nanda2011dynamic,
  title={Dynamic proportional hazard rate and reversed hazard rate models},
  author={Nanda, Asok K and Das, Suchismita},
  journal={Journal of Statistical Planning and Inference},
  volume={141},
  number={6},
  pages={2108--2119},
  year={2011},
  publisher={Elsevier}
}

@article{balakrishnan2018necessary,
  title={Necessary and sufficient conditions for stochastic orders between (n-r+1)-out-of-n systems in proportional hazard (reversed hazard) rates model},
  author={Balakrishnan, Narayanaswamy and Barmalzan, Ghobad and Haidari, Abedin and Najafabadi, Amir T Payandeh},
  journal={Communications in Statistics-Theory and Methods},
  volume={47},
  number={23},
  pages={5854--5866},
  year={2018},
  publisher={Taylor \& Francis}
}

@article{gupta1998modeling,
  title={Modeling failure time data by Lehman alternatives},
  author={Gupta, Ramesh C and Gupta, Pushpa L and Gupta, Rameshwar D},
  journal={Communications in Statistics-Theory and methods},
  volume={27},
  number={4},
  pages={887--904},
  year={1998},
  publisher={Taylor \& Francis}
}

@article{gupta2007proportional,
  title={Proportional reversed hazard rate model and its applications},
  author={Gupta, Ramesh C and Gupta, Rameshwar D},
  journal={Journal of statistical planning and inference},
  volume={137},
  number={11},
  pages={3525--3536},
  year={2007},
  publisher={Elsevier}
}

@book{andersen2012statistical,
  title={Statistical models based on counting processes},
  author={Andersen, Per K and Borgan, Ornulf and Gill, Richard D and Keiding, Niels},
  year={2012},
  publisher={Springer Science \& Business Media}
}

@article{maadani2020new,
  title={A new generalized varentropy and its properties},
  author={Maadani, Saeid and GR, Mohtashami Borzadaran and AH, Rezaei Roknabadi},
  journal={Ural Mathematical Journal},
  volume={6},
  number={1 (10)},
  pages={114--129},
  year={2020},
  publisher={Федеральное государственное бюджетное учреждение науки {\guillemotleft}Институт математики~…}
}

@article{rosenblat1956remarks,
  title={Remarks on some nonparametric estimates of a density function},
  author={Rosenblat, M},
  journal={The Annals of Mathematica Statistics},
  volume={27},
  pages={832--837},
  year={1956}
}

@article{parzen1962estimation,
  title={On estimation of a probability density function and mode},
  author={Parzen, Emanuel},
  journal={The {A}nnals of {M}athematical {S}tatistics},
  volume={33},
  number={3},
  pages={1065--1076},
  year={1962},
  publisher={JSTOR}
}

@inproceedings{fradelizi2016optimal,
	title={Optimal concentration of information content for log-concave densities},
	author={Fradelizi, Matthieu and Madiman, Mokshay and Wang, Liyao},
	booktitle={High Dimensional Probability VII: The Carg{\`e}se Volume},
	pages={45--60},
	year={2016},
	organization={Springer}
}

@article{kontoyiannis1997second,
  title={Second-order noiseless source coding theorems},
  author={Kontoyiannis, Ioannis},
  journal={IEEE Transactions on Information Theory},
  volume={43},
  number={4},
  pages={1339--1341},
  year={1997},
  publisher={IEEE}
}

@article{kontoyiannis2013optimal,
  title={Optimal lossless data compression: Non-asymptotics and asymptotics},
  author={Kontoyiannis, Ioannis and Verd{\'u}, Sergio},
  journal={IEEE Transactions on Information Theory},
  volume={60},
  number={2},
  pages={777--795},
  year={2014},
  publisher={IEEE}
}

@article{cacoullos1989characterizations,
	title={Characterizations of distributions by variance bounds},
	author={Cacoullos, T and Papathanasiou, V},
	journal={Statistics \& Probability Letters},
	volume={7},
	number={5},
	pages={351--356},
	year={1989},
	publisher={Elsevier}
}

@article{bobkov2011concentration,
	title={Concentration of the information in data with log-concave distributions},
	author={Bobkov, Sergey and Madiman, Mokshay},
	journal={Annals of Probability},
	volume={39},
	pages={1528-1543},
	year={2011}
}

@article{ebrahimi1995new,
  title={New partial ordering of survival functions based on the notion of uncertainty},
  author={Ebrahimi, Nader and Pellerey, Franco},
  journal={Journal of Applied probability},
  volume={32},
  number={1},
  pages={202--211},
  year={1995},
  publisher={Cambridge University Press}
}

@article{di2021analysis,
  title={Analysis and applications of the residual varentropy of random lifetimes},
  author={Di Crescenzo, Antonio and Paolillo, Luca},
  journal={Probability in the Engineering and Informational Sciences},
  volume={35},
  number={3},
  pages={680--698},
  year={2021},
  publisher={Cambridge University Press}
}

@article{buono2022varentropy,
  title={Varentropy of past lifetimes},
  author={Buono, Francesco and Longobardi, Maria and Pellerey, Franco},
  journal={Mathematical Methods of Statistics},
  volume={31},
  number={2},
  pages={57--73},
  year={2022},
  publisher={Springer}
}

@article{raqab2022varentropy,
  title={Varentropy of inactivity time of a random variable and its related applications},
  author={Raqab, Mohammad Z and Bayoud, Husam A and Qiu, Guoxin},
  journal={IMA Journal of Mathematical Control and Information},
  volume={39},
  number={1},
  pages={132--154},
  year={2022},
  publisher={Oxford University Press}
}

@article{maadani2022varentropy,
  title={Varentropy of order statistics and some stochastic comparisons},
  author={Maadani, S and Mohtashami Borzadaran, Gholam Reza and Rezaei Roknabadi, AH},
  journal={Communications in Statistics-Theory and Methods},
  volume={51},
  number={18},
  pages={6447--6460},
  year={2022},
  publisher={Taylor \& Francis}
}

@article{navarro2013stochastic,
  title={Stochastic ordering properties for systems with dependent identically distributed components},
  author={Navarro, Jorge and del {\'A}guila, Yolanda and Sordo, Miguel A and Su{\'a}rez-Llorens, Alfonso},
  journal={Applied Stochastic Models in Business and Industry},
  volume={29},
  number={3},
  pages={264--278},
  year={2013},
  publisher={Wiley Online Library}
}

@article{burbea1982convexity,
	title={On the convexity of some divergence measures based on entropy functions},
	author={Burbea, Jacob and Rao, C},
	journal={IEEE Transactions on Information Theory},
	volume={28},
	number={3},
	pages={489--495},
	year={1982},
	publisher={IEEE}
}

@article{klein2016cumulative,
  title={Cumulative paired $\varphi$-entropy},
  author={Klein, Ingo and Mangold, Benedikt and Doll, Monika},
  journal={Entropy},
  volume={18},
  number={7},
  pages={248},
  year={2016},
  publisher={MDPI}
}

@article{di2002entropy,
  title={Entropy-based measure of uncertainty in past lifetime distributions},
  author={Di Crescenzo, Antonio and Longobardi, Maria},
  journal={Journal of Applied probability},
  volume={39},
  number={2},
  pages={434--440},
  year={2002},
  publisher={Cambridge University Press}
}

@article{di2009cumulative,
  title={On cumulative entropies},
  author={Di Crescenzo, Antonio and Longobardi, Maria},
  journal={Journal of Statistical Planning and Inference},
  volume={139},
  number={12},
  pages={4072--4087},
  year={2009},
  publisher={Elsevier}
}

@article{cali2020properties,
  title={Properties for generalized cumulative past measures of information},
  author={Cal{\`\i}, Camilla and Longobardi, Maria and Navarro, Jorge},
  journal={Probability in the Engineering and Informational Sciences},
  volume={34},
  number={1},
  pages={92--111},
  year={2020},
  publisher={Cambridge University Press}
}

@article{di2007weighted,
	title={On weighted residual and past entropies},
	author={Di Crescenzo, Antonio and Longobardi, Maria},
	journal={Scientiae Mathematicae Japonicae},
    volume={64},
    number={3},
    pages={255--266},
    year={2006}
   }

@article{shannon1948mathematical,
  title={A mathematical theory of communication},
  author={Shannon, Claude Elwood},
  journal={The Bell System Technical Journal},
  volume={27},
  number={3},
  pages={379--423},
  year={1948},
  publisher={Nokia Bell Labs}
}

@article{kharazmi2021informational,
  title={Informational properties of transmuted distributions},
  author={Kharazmi, Omid and Balakrishnan, Narayanaswamy},
  journal={Filomat},
  volume={35},
  number={13},
  pages={4287--4303},
  year={2021}
}

@article{rao2004cumulative,
  title={Cumulative residual entropy: a new measure of information},
  author={Rao, Murali and Chen, Yunmei and Vemuri, Baba C and Wang, Fei},
  journal={IEEE transactions on Information Theory},
  volume={50},
  number={6},
  pages={1220--1228},
  year={2004},
  publisher={IEEE}
}

@article{di2021fractional,
  title={Fractional generalized cumulative entropy and its dynamic version},
  author={Di Crescenzo, Antonio and Kayal, Suchandan and Meoli, Alessandra},
  journal={Communications in Nonlinear Science and Numerical Simulation},
  volume={102},
  pages={105899},
  year={2021},
  publisher={Elsevier}
}

\end{document}